\documentclass[oneside,english]{amsart}
\usepackage[T1]{fontenc}
\usepackage[latin9]{inputenc}
\usepackage{geometry}
\usepackage{color}
\usepackage{verbatim}
\usepackage{amstext}
\usepackage{amsthm}
\usepackage{amssymb}
\usepackage{esint}

\makeatletter
\numberwithin{equation}{section}
\numberwithin{figure}{section}
\theoremstyle{plain}
\newtheorem{thm}{\protect\theoremname}
\theoremstyle{definition}
\newtheorem{defn}[thm]{\protect\definitionname}
\theoremstyle{remark}
\newtheorem{rem}[thm]{\protect\remarkname}
\theoremstyle{plain}
\newtheorem{prop}[thm]{\protect\propositionname}
\theoremstyle{plain}
\newtheorem{lem}[thm]{\protect\lemmaname}
\theoremstyle{plain}
\newtheorem{cor}[thm]{\protect\corollaryname}

\makeatother

\usepackage{babel}
\providecommand{\corollaryname}{Corollary}
\providecommand{\definitionname}{Definition}
\providecommand{\lemmaname}{Lemma}
\providecommand{\propositionname}{Proposition}
\providecommand{\remarkname}{Remark}
\providecommand{\theoremname}{Theorem}

\makeatletter
\@namedef{subjclassname@2020}{
\textup{2020} Mathematics Subject Classification}
\makeatother

\begin{document}
\title[Global attractors and convergence for Cahn-Hilliard equation on conic manifolds]{Existence of global attractors and convergence of solutions for the Cahn-Hilliard equation on manifolds with conical singularities}

\author{Pedro T. P. Lopes}
\address{Instituto de Matem\'atica e Estat\'istica, Universidade de S\~ao Paulo, Rua do Mat\~ao 1010, 05508-090, S\~ao Paulo, SP, Brazil}
\email{pplopes@ime.usp.br}

\author{Nikolaos Roidos}
\address{Department of Mathematics, University of Patras, 26504 Rio Patras, Greece}
\email{roidos@math.upatras.gr}

\subjclass[2020]{35B40; 35K58; 35K65; 35K91; 35R01}

\begin{abstract}
We consider the Cahn-Hilliard equation on manifolds with conical singularities and prove existence of global attractors in higher order Mellin-Sobolev spaces with asymptotics. We also show convergence of solutions in the same spaces to an equilibrium point and provide asymptotic behavior of the equilibrium near the conical tips in terms of the local geometry.
\end{abstract}

\maketitle

\section{Introduction}

In this article, we show existence and regularity of global attractors
as well as convergence results for the Cahn-Hilliard equation considered
on {\em manifolds with conical singularities}. We model such a manifold
as a $(n+1)$-dimensional compact manifold $\mathcal{B}$ with closed
boundary $\partial\mathcal{B}$, $n\ge1$, which is endowed with a
degenerated Riemannian metric $g$ that, in local coordinates $(x,y)\in[0,1)\times\partial\mathcal{B}$
on a collar neighborhood of the boundary, has the following expression
$$
g=dx^{2}+x^{2}h(x),
$$
where $[0,1)\ni x\mapsto h(x)$
is a smooth family of Riemannian metrics on $\partial\mathcal{B}$.
We denote $\mathbb{B}=(\mathcal{B},g)$ and $\partial\mathbb{B}=(\partial\mathcal{B},h(0))$.
The Laplacian on $\mathbb{B}$, in local coordinates $(x,y)\in[0,1)\times\partial\mathcal{B}$
on the collar part, admits the following degenerate expression
$$
\Delta=\frac{1}{x^{2}}\Big((x\partial_{x})^{2}+(n-1+\frac{x\partial_{x}\det(h(x))}{2\det(h(x))})(x\partial_{x})+\Delta_{h(x)}\Big),
$$
where $\Delta_{h(x)}$ is the Laplacian on $(\partial\mathcal{B},h(x))$. The operator $\Delta$ belongs to the class of {\em cone differential operators} or {\em Fuchs type operators}, see Section \ref{sec:Realizations} for more details.

On $\mathbb{B}$ we consider the following problem
\begin{equation}
\begin{aligned}u'(t)+\Delta^{2}u(t) & =\Delta(u^{3}(t)-u(t)),\quad t\in(0,T),\\
u(0) & =u_{0},
\end{aligned}
\label{eq:mainequation}
\end{equation}
which is known as Cahn-Hilliard equation (CH for short); it is a diffusion
interface equation that models phase separation of a binary mixture.
On classical domains, (\ref{eq:mainequation}) has been generalized
and extensively studied in many directions and aspects, such as existence,
regularity and convergence of solutions, existence of global attractors,
etc. A sufficient number of related results can be found in \cite{Miranville}.

However, on singular domains much less is known. Using the theory
of cone differential operators, in \cite{RS1} it was first shown
short-time existence of CH on $\mathbb{B}$ for the case where $h(\cdot)$
is constant, by employing $L^{p}$-maximal regularity techniques. Those results
were extended to arbitrary $\mathbb{B}$ and improved to higher regularity
in \cite{RS2}. Finally, global solutions and smoothing results were
proved in \cite{LopesRoidos}. Summarizing those results, let us assume that $\dim(\mathbb{B})=n+1\in\{2,3\}$, choose $s\ge0$ and let the exponent $\gamma$ be as follows 
\begin{gather}
\frac{\dim(\mathbb{B})-4}{2}<\gamma<\min\Big\{-1+\sqrt{\Big(\frac{\dim(\mathbb{B})-2}{2}\Big)^{2}-\lambda_{1}},\frac{\text{dim}(\mathbb{B})-4}{4}\Big\},\label{gamma}
\end{gather}
where $0=\lambda_{0}>\lambda_{1}>\cdots$ are the eigenvalues of the boundary
Laplacian $\Delta_{h(0)}$ on $\partial\mathbb{B}$. Denote by $\mathcal{H}^{\eta,\rho}(\mathbb{B})$, $\eta, \rho\in\mathbb{R}$, the {\em Mellin-Sobolev space}, see Definition \ref{def:mellinsobolevspaces}. Moreover, let $\mathbb{R}_{\omega}$ and $\mathbb{C}_{\omega}$ be the finite dimensional spaces of
smooth functions on $\mathbb{B}$ that are locally constant close
to the singularities, with values in $\mathbb{R}$ and $\mathbb{C}$
respectively, see Section \ref{sec:Realizations} for details. Then, for any
real-valued $u_{0}\in\mathcal{H}^{s+2,\gamma+2}(\mathbb{B})\oplus\mathbb{R}_{\omega}$,
there exists a unique global solution in the following sense: for
any $T>0$ there exists a unique $u\in H^{1}(0,T;\mathcal{H}^{s,\gamma}(\mathbb{B}))\cap L^{2}(0,T;\mathcal{D}(\Delta_{s}^{2}))$
solving (\ref{eq:mainequation}) on $[0,T]\times\mathbb{B}$. Furthermore, the solution $u$ satisfies the regularity
\begin{gather}
u\in\bigcap_{s\ge0}C^{\infty}((0,\infty);\mathcal{D}(\Delta_{s}^{2}))\label{extrareg}
\end{gather}
and
\begin{equation}
u\in C([0,\infty);\mathcal{H}^{s+2,\gamma+2}(\mathbb{B})\oplus\mathbb{R}_{\omega})\hookrightarrow C([0,\infty);C(\mathbb{B})).\label{contsolu}
\end{equation}

The bi-Laplacian domain we choose is
\begin{equation}\label{eq:defbilap}
\mathcal{D}(\Delta_{s}^{2})=\{u\in\mathcal{H}^{s+2,\gamma+2}(\mathbb{B})\oplus\mathbb{C}_{\omega}:\Delta u\in\mathcal{H}^{s+2,\gamma+2}(\mathbb{B})\oplus\mathbb{C}_{\omega}\},\end{equation}
where $\gamma$ is always as \eqref{gamma}. It satisfies 
\begin{gather}
\mathcal{D}(\Delta_{s}^{2})=\mathcal{D}(\Delta_{s,\min}^{2})\oplus\mathbb{C}_{\omega}\oplus\mathcal{E}_{\Delta^{2},\gamma}.\label{bilapintro1}
\end{gather}
Here $\mathcal{E}_{\Delta^{2},\gamma}$ is an $s$-independent finite dimensional space consisting of $C^{\infty}(\mathbb{B}^{\circ})$-functions, that in
local coordinates $(x,y)\in[0,1)\times\partial\mathcal{B}$, 
take the form $\omega(x)c(y)x^{\rho}\ln^{k}(x)$, $\rho\in\mathbb{C}$, $k\in\{0,1,2,3\}$, where $\mathbb{B}^{\circ}=\mathcal{B}\backslash\partial\mathcal{B}$ and $c\in C^{\infty}(\partial\mathbb{B})$. More precisely, there exists a discrete set of points $Z_{\Delta^2}$ in $\mathbb{C}$, determined only by the family of metrics $h(\cdot)$, such that the exponents $\rho$ coincide with the set $Z_{\Delta^2}\cap \{z\in\mathbb{C} : \mathrm{Re}(z)\in[\frac{n-7}{2}-\gamma,\frac{n-3}{2}-\gamma)\}$. The exponents $k$ are also determined by $h(\cdot)$. In particular, when $h(\cdot)=h$ is constant, the set $Z_{\Delta^2}$ and the exponents $k$ associated to each $\rho\in Z_{\Delta^2}$, are determined by $n$ and the spectrum of $\Delta_{h}$. The minimal domain $\mathcal{D}(\Delta_{s,\min}^{2})$ stands for the domain of the closure of $\Delta^2: C_{c}^{\infty}(\mathbb{B}^{\circ})\rightarrow \mathcal{H}^{s,\gamma}(\mathbb{B})$, and satisfies
\begin{gather}
\mathcal{H}^{s+4,\gamma+4}(\mathbb{B})\hookrightarrow\mathcal{D}(\Delta_{s,\min}^{2})\hookrightarrow\bigcap_{\varepsilon>0}\mathcal{H}^{s+4,\gamma+4-\varepsilon}(\mathbb{B}),\label{bilapintro2}
\end{gather}
while 
\begin{gather}\label{bilapintro3}
\mathcal{D}(\Delta_{s,\min}^{2})=\mathcal{H}^{s+4,\gamma+4}(\mathbb{B}) 
\end{gather}
provided that
\begin{gather}\label{bilapintro4}
\{\gamma+1,\gamma+3\}\cap\bigcup_{\lambda_{j}\in \sigma(\Delta_{h(0)})}\Big\{\pm\sqrt{\Big(\frac{\mathrm{dim}(\mathbb{B})-2}{2}\Big)^{2}-\lambda_{j}}\Big\}=\emptyset.
\end{gather}
Consequently, both spaces $\mathcal{D}(\Delta_{s}^{2})$ and $\mathcal{E}_{\Delta^{2},\gamma}$ are determined explicitly by $h(\cdot)$ and $\gamma$, see Corollary \ref{cor:(Bi-Laplacian)} for details. 

These results allow us to define for any $s\ge0$ a {\em semiflow} $T:[0,\infty)\times \mathcal{H}^{s+2,\gamma+2}(\mathbb{B})\oplus\mathbb{R}_{\omega}\to \mathcal{H}^{s+2,\gamma+2}(\mathbb{B})\oplus\mathbb{R}_{\omega}$
on real valued-functions by $T(t)u_{0}:=T(t,u_0)=u(t)$, see e.g. \cite[Chapter 1, Section 1.1]{Temam} for more details
on semiflows, also known as {\em semigroups}. Let $X_{1,0}^{s}$ be the space of all real-valued functions $u\in\mathcal{H}^{s+2,\gamma+2}(\mathbb{B})\oplus\mathbb{R}_{\omega}$
such that $\int_{\mathbb{B}}ud\mu_{g}=0$, where $d\mu_{g}$ is the
measure associated with the metric $g$. Then $T$ can be restricted
to $X_{1,0}^{s}$, see Section \ref{sec:Existence-and-regularity-1}. Our main results are the
following.

\begin{thm}
\label{thm:MainTheorem} Let $s\ge0$, $\gamma$ be as \eqref{gamma} and $\mathcal{D}(\Delta_{s}^{2})$ be the bi-Laplacian domain described in \eqref{eq:defbilap}-\eqref{bilapintro4}.\\
\emph{(i) (Global attractor)} The semiflow $T:[0,\infty)\times X_{1,0}^{s}\to X_{1,0}^{s}$ has
an $s$-independent global attractor
$\mathcal{A}\subset\cap_{r>0}\mathcal{D}(\Delta_{r}^{2})$. Moreover, if $B$ is a bounded set of $X_{1,0}^{s}$, then for any $r>0$, $T(t)B$ is, for sufficiently large $t$, a bounded set of $\mathcal{D}(\Delta_{r}^{2})$ and
$$
\lim_{t\to\infty}(\sup_{b\in B}\inf_{a\in\mathcal{A}}\left\Vert T(t)b-a\right\Vert _{\mathcal{D}(\Delta_{r}^{2})})=0.
$$
\emph{(ii) (Convergence to equilibrium)} If $u_{0}\in X_{1,0}^{0}$, then there exists a $u_{\infty}\in\cap_{r>0}\mathcal{D}(\Delta_{r}^{2})$ such that $\lim_{t\to\infty}T(t)u_{0}=u_{\infty}$, where the convergence occurs in $\mathcal{D}(\Delta_{r}^{2})$ for each $r\ge0$.
\end{thm}

The definition of global attractor is recalled in Section \ref{sec:Existence-and-regularity-1}.
For proving part (i) of Theorem \ref{eq:mainequation}, we follow
the strategy of Temam \cite{Temam} to obtain estimates in a lower regularity space
$H_{0}^{-1}(\mathbb{B})$, see Definition \ref{def:H-10}, and of \cite{Song}
for obtaining higher regularity. For convergence to equilibrium, we first
obtain the Lojasiewicz-Simon inequality due to \cite{Simon}, and proceed
as \cite{Chill}, \cite{HJ} and \cite{Piotr}.

Concerning real-life applications of the above approach, recall first
that the physical effects described by CH, as well as other evolution
equations, occur in reality in many different types of domains and
surfaces (manifolds), which are usually not smooth: many of them have
edges, conical points, cusps, or even combinations of these and other
types of singularities. In this context, conic manifolds are fundamental
and a natural place to start. They describe simple point singularities,
which, apart from their intrinsic interest, can be used to build more
general ones \cite{Schu2}, \cite{Schu}.

Moreover, whenever we are studying a smooth $(n+1)$-dimensional Riemannian
manifold $\mathcal{M}$ endowed with a Riemannian metric $f$, an
important question is: \emph{how does the local geometry on $\mathbb{M}=(\mathcal{M},f)$
affect the evolution?} An answer to this question arises as follows:
fix a point $o$ on $\mathcal{M}$ and denote by $d(o,z)$ the geodesic
distance between $o$ and $z\in\mathcal{M}\backslash\{o\}$, induced
by the metric $f$. There exists an $r>0$ such that $(x,y)\in(0,r)\times\mathbb{S}^{n}$
are local coordinates around $o$ and moreover, the metric in these
coordinates becomes $f=dx^{2}+x^{2}f_{\mathbb{S}^{n}}(x)$, where
$\mathbb{S}^{n}=\{z\in\mathbb{R}^{n+1}\,:\,|z|=1\}$ is the unit sphere
and $x\mapsto f_{\mathbb{S}^{n}}(x)$ is a smooth family of Riemannian
metrics on $\mathbb{S}^{n}$. In case of $f_{\mathbb{S}^{n}}(\cdot)$
being smooth up to $x=0$, we can regard $((\mathcal{M}\backslash\{0\})\cup(\{0\}\times\mathbb{S}^{n}),f)$
as a conic manifold with one isolated conical singularity at $o$.
On the other hand, since our problem involves the Laplacian, it becomes
now degenerate. However, an application of our results shows that
the asymptotic expansion of the solutions near $o$ is provided by
the expansion \eqref{bilapintro1}, where the boundary
Laplacian $\Delta_{h(x)}$ now has to be replaced by the Laplacian
$\Delta_{f_{\mathbb{S}^{n}}(x)}$ on $(\mathbb{S}^{n},f_{\mathbb{S}^{n}}(x))$.
Hence, in particular, through the structure of the spaces $\mathcal{E}_{\Delta^{2},\gamma}$,
we obtain an interplay between the spectrum of $\Delta_{f_{\mathbb{S}^{n}}(x)}$
and the evolution.

Though the strategies are mostly well established, the technical results that allow us to use them in the context of conical singularities are not, and, therefore, the strategies have to be adapted to this situation. For this reason new results on interpolation and embedding of Mellin-Sobolev spaces are developed in this article.

In Section 2, we define suitable function spaces to work on conic
manifolds and study their embeddings. Section 3
is devoted to the domain description and the properties of the Laplacian
and bi-Laplacian and to provide some facts about the complex interpolation of those spaces. Part (i) of Theorem \ref{thm:MainTheorem} is
proved in Section 4 and part (ii) in Section 5.

\section{Function spaces\label{sec:Function-spaces}}

Fix a smooth non-negative
function $\omega\in C^{\infty}(\mathbb{B})$ supported on the collar
neighborhood $(x,y)\in[0,1)\times\partial\mathcal{B}$ such that $\omega$
depends only on $x$ and $\omega=1$ near $\{0\}\times\partial\mathcal{B}$.
Moreover denote by $C_{c}^{\infty}$ the space of smooth compactly
supported functions and by $H^{s}$, $s\in\mathbb{R}$, the usual
Bessel potential spaces defined using the $L^{2}$-norm.
\begin{defn}
[Mellin-Sobolev spaces] \label{def:mellinsobolevspaces} Let $\gamma\in\mathbb{R}$
and consider the map 
$$
M_{\gamma}:C_{c}^{\infty}(\mathbb{R}_{+}\times\mathbb{R}^{n})\rightarrow C_{c}^{\infty}(\mathbb{R}^{n+1})\quad\mbox{defined by}\quad u(x,y)\mapsto e^{(\gamma-\frac{n+1}{2})x}u(e^{-x},y).
$$
Let $\kappa_{j}:U_{j}\subseteq\partial\mathcal{B}\rightarrow\mathbb{R}^{n}$,
$j\in\{1,...,N\}$, $N\in\mathbb{N}\backslash\{0\}$, $\mathbb{N}:=\{0,1,2,...\}$, be a covering
of $\partial\mathcal{B}$ by coordinate charts and let $\{\phi_{j}\}_{j\in\{1,...,N\}}$
be a subordinated partition of unity. For any $s,\gamma\in\mathbb{R}$,
the Mellin Sobolev space $\mathcal{H}^{s,\gamma}(\mathbb{B})$ is
defined to be the space of all distributions $u$ on the interior $\mathbb{B}^{\circ}$
such that the norm
\begin{equation}
\|u\|_{\mathcal{H}^{s,\gamma}(\mathbb{B})}=\sum_{j=1}^{N}\|M_{\gamma}(1\otimes\kappa_{j})_{\ast}(\omega\phi_{j}u)\|_{H^{s}(\mathbb{R}^{n+1})}+\|(1-\omega)u\|_{H^{s}(2\mathbb{B})}\label{mellinsobolev}
\end{equation}
is defined and finite, where $2\mathbb{B}$ is the double of $\mathbb{B}$
and $\ast$ refers to the push-forward of distributions. Different
choices of $\omega$, covering and partition of unity give us the
same spaces with equivalent norms. The space $\mathcal{H}^{s,\gamma}(\mathbb{B})$ is a Banach algebra, up to an equivalent norm, whenever $s>(n+1)/2$ and $\gamma\ge(n+1)/2$, see \cite[Lemma 3.2]{RS3}.

If $s\in\mathbb{N}$, then $\mathcal{H}^{s,\gamma}(\mathbb{B})$
coincides with the space of all functions $u$ in $H_{\text{loc}}^{s}(\mathbb{B}^{\circ})$
that satisfy 
\begin{equation}
x^{\frac{n+1}{2}-\gamma}(x\partial_{x})^{k}\partial_{y}^{\alpha}(\omega(x)u(x,y))\in L^{2}([0,1)\times\partial\mathcal{B},\sqrt{\det(h(x))}\frac{dx}{x}dy),\quad k+|\alpha|\le s.\label{mellinsobolevinteger}
\end{equation}
\end{defn}

In Section \ref{sec:Realizations}, we will associate the Mellin-Sobolev spaces with the Laplacian and bi-Laplacian.

\begin{rem}
\label{rem:xmellinL2weight}Let $\mathsf{x}:\mathbb{B}\to[0,1]$ be
a smooth positive function on $\mathbb{B}^{\circ}$ that is equal
to $\mathsf{x}(x,y)=x$ on the collar neighborhood $[0,1)\times\partial\mathcal{B}$.
Then $u\in\mathcal{H}^{0,\gamma}(\mathbb{B})$ iff $\mathsf{x}^{-\gamma}u\in L^{2}(\mathbb{B})$,
where $L^{2}(\mathbb{B})=\mathcal{H}^{0,0}(\mathbb{B})$. We define
the spaces $L^{p}(\mathbb{B})$ using the measure $d\mu_{g}$ induced
by the metric $g$. Note that $d\mu_{g}=\sqrt{\det(h(x))}x^{n}dxdy$
on the collar neighborhood. Moreover, for any $\alpha\in\mathbb{R}$, let $x^{\alpha}L^{p}(\mathbb{B}):=\{u:\int_{\mathbb{B}}|x^{-\alpha}u|^{p}d\mu_{g}<\infty\}$. Finally, recall that the inner product in $\mathcal{H}^{0,0}(\mathbb{B})$ induces an identification of the dual space of $\mathcal{H}^{s,\gamma}(\mathbb{B})$ with $\mathcal{H}^{-s,-\gamma}(\mathbb{B})$, see e.g. \cite[Lemma 3.2 (ii)]{LopesRoidos}.
\end{rem}

Besides the Mellin-Sobolev spaces, we define the following space.
\begin{defn}
Let $H^{1}(\mathbb{B})$ be the completion of $C_{c}^{\infty}(\mathbb{B}^{\circ})$
with respect to the inner product
$$
(u,v)_{H^{1}(\mathbb{B})}=\int_{\mathbb{B}}u\overline{v}d\mu_{g}+\int_{\mathbb{B}}\left\langle \nabla u,\overline{\nabla v}\right\rangle d\mu_{g},
$$
where $\nabla$ and $\left\langle \cdot,\cdot\right\rangle $ are defined
by the conical metric $g$.
\end{defn}

We investigate now certain properties of the space $H^{1}(\mathbb{B})$
and its connection with $\mathcal{H}^{s,\gamma}(\mathbb{B})$.
\begin{rem}
For the following computations, we note that\\
(1) The boundedness of $\int_{\mathbb{B}}\left|u\right|^{2}d\mu_{g}$
is equivalent to
$$
u\in L_{\text{loc}}^{2}(\mathbb{B}^{\circ})\,\text{and}\,(x,y)\mapsto x^{\frac{n+1}{2}}\omega(x)u(x,y)\in L^{2}([0,1]\times\partial\mathcal{B},\sqrt{\det(h(x))}\frac{dx}{x}dy).
$$
(2) If $\int_{\mathbb{B}}\left|u\right|^{2}d\mu_{g}<\infty$, then the boundedness of $\int_{\mathbb{B}}\left\langle \nabla u,\nabla\overline{u}\right\rangle d\mu_{g}$ is equivalent to
$$
u\in H_{\text{loc}}^{1}(\mathbb{B}^{\circ})\,\text{and}\,(x,y)\mapsto x^{\frac{n-1}{2}}(x\partial_{x})^{k}\partial_{y}^{\alpha}(\omega(x)u(x,y))\in L^{2}([0,1]\times\partial\mathcal{B},\sqrt{\det(h(x))}\frac{dx}{x}dy),\,k+\left|\alpha\right|=1.
$$

The last statement can be easily proved once we recall that in local
coordinates of $[0,1)\times\partial\mathcal{B}$ we have
$$
\left\langle \nabla u,\nabla\overline{v}\right\rangle =x^{-2}(x\partial_{x}u)(x\partial_{x}\overline{v})+x^{-2}\sum_{i,j=1}^{n}h^{ij}(x,y)(\partial_{y_{i}}u)(\partial_{y_{j}}\overline{v}).
$$
\end{rem}

Along this paper, we use $\hookrightarrow$ and $\overset{c}{\hookrightarrow}$ to denote continuous and compact embedding, respectively. We recall that $\mathcal{H}^{s,\gamma}(\mathbb{B})\hookrightarrow\mathcal{H}^{s',\gamma'}(\mathbb{B})$, when $s\ge s'$ and $\gamma\ge\gamma'$, and $\mathcal{H}^{s,\gamma}(\mathbb{B})\overset{c}{\hookrightarrow}\mathcal{H}^{s',\gamma'}(\mathbb{B})$, when $s>s'$ and $\gamma>\gamma'$, see \cite[Theorem 2.1.53]{Schu2}.

\begin{prop}
\label{prop:inclusionH1mellin} For
any $\beta<1$, the following inclusions hold
$$
\mathcal{H}^{1,1}(\mathbb{B})\oplus\mathbb{C}_{\omega}\hookrightarrow H^{1}(\mathbb{B})\hookrightarrow\mathcal{H}^{1,\beta}(\mathbb{B}),
$$
where $u\oplus v$ is identified with $u+v$ and the first inclusion is just $u\oplus v \mapsto u+v$. In particular, $H^{1}(\mathbb{B})\overset{c}{\hookrightarrow}{H}^{0,0}(\mathbb{B})$.
\end{prop}

\begin{proof}
We proceed in several steps. Let us denote by $C$ positive
constants that can change along the proof. For simplicity we ignore the term $\sqrt{\text{\ensuremath{\det}}(h(x))}$ in the proof, as it is uniformly bounded from above and below, and we abuse the notation $\int_{\partial\mathbb{B}}$ since the computations are made in local coordinates. We also note that it suffices to check the inclusion on the collar neighborhood and for functions $u\in C_c^\infty(\mathbb{B}^\circ)$.\\
\emph{Step 1:} $\mathcal{H}^{1,1}(\mathbb{B})\hookrightarrow H^{1}(\mathbb{B})$. We have
$$
\int_{0}^{1}\int_{\partial\mathbb{B}}\left|\omega(x)u(x,y)\right|^{2}x^{n}dxdy\le\int_{0}^{1}\int_{\partial\mathbb{B}}\left|x^{\frac{n+1}{2}-1}\omega(x)u(x,y)\right|^{2}\frac{dx}{x}dy.
$$
Hence it is clear that
$$
\int_{\mathbb{B}}\left|u\right|^{2}d\mu_{g}\le C\left\Vert u\right\Vert _{\mathcal{H}^{1,1}(\mathbb{B})}^{2}.
$$
Moreover
\begin{eqnarray*}\lefteqn{\int_{0}^{1}\int_{\partial\mathbb{B}}\frac{1}{x^{2}}\left(\left|x\partial_{x}(\omega u)\right|^{2}+\sum_{i,j=1}^{n}h^{ij}(x,y)\partial_{y_{i}}(\omega u)\partial_{y_{j}}(\omega\overline{u})\right)x^{n}dxdy}\\
 & \le& C\sum_{k+\left|\alpha\right|=1}\int_{0}^{1}\int_{\partial\mathbb{B}}\left|x^{\frac{n-1}{2}}(x\partial_{x})^{k}\partial_{y}^{\alpha}(\omega(x)u(x,y))\right|^{2}\frac{dx}{x}dy,
\end{eqnarray*}
which implies that
$$
\int_{\mathbb{B}}\left\langle \nabla u,\overline{\nabla u}\right\rangle d\mu_{g}\le C\left\Vert u\right\Vert _{\mathcal{H}^{1,1}(\mathbb{B})}^{2}.
$$
We conclude that
$$
\left\Vert u\right\Vert _{H^{1}(\mathbb{B})}\le C\left\Vert u\right\Vert _{\mathcal{H}^{1,1}(\mathbb{B})}.
$$
\emph{Step 2}: For each $\varepsilon>0$, we have $\mathsf{x}^{\varepsilon}\in H^{1}(\mathbb{B})$. Let $0<r<1$ and $\chi_{r}:\mathbb{B}^{\circ}\to[0,1]$ be such that $\chi_{r}(x,y)=1-\omega(x/r)$, for $(x,y)\in[0,1)\times\partial\mathcal{B}$
and $\chi_{r}$ be equal to 1 outside the collar neighborhood. It
is enough to prove that $\lim_{r\to0}\chi_{r}\mathsf{x}^{\varepsilon}=\mathsf{x}^{\varepsilon}$
in $H^{1}(\mathbb{B})$, as $\chi_{r}\mathsf{x}^{\varepsilon}\in C_{c}^{\infty}(\mathbb{B}^{\circ})$.
For this, we must prove that
$$
\begin{aligned}\text{(i)}\quad & \lim_{r\to0}\int_{0}^{1}\int_{\partial\mathbb{B}}\left|\omega(\chi_{r}x^{\varepsilon}-x^{\varepsilon})\right|^{2}x^{n}dxdy =0,\\
\text{(ii)}\quad & \lim_{r\to0}\int_{0}^{1}\int_{\partial\mathbb{B}}\left|\partial_{y_{j}}(\omega\chi_{r}x^{\varepsilon})-\partial_{y_{j}}(\omega x^{\varepsilon})\right|^{2}x^{n-2}dxdy =0,\\
\text{(iii)} \quad & \lim_{r\to0}\int_{0}^{1}\int_{\partial\mathbb{B}}\left|x\partial_{x}(\omega\chi_{r}x^{\varepsilon})-x\partial_{x}(\omega x^{\varepsilon})\right|^{2}x^{n-2}dxdy =0.
\end{aligned}
$$
Note that (i) follows directly from the dominated convergence
theorem and (ii) is identically zero. For (iii), we have that the integral is smaller or equal to two times
$$
\int_{0}^{1}\int_{\partial\mathbb{B}}\left|\chi_{r}x\partial_{x}(\omega x^{\varepsilon})-x\partial_{x}(\omega x^{\varepsilon})\right|^{2}x^{n-2}dxdy+\int_{0}^{1}\int_{\partial\mathbb{B}}\left|\omega(x)x^{\varepsilon}\partial_{x}\chi_{r}\right|^{2}x^{n}dxdy.
$$
Only the last term is important, as we can handle the first one directly with dominated convergence theorem. Note that $\left|\omega(x)x^{\varepsilon}\partial_{x}\chi_{r}\right|^{2}x^{n}=\left|\omega(x)x/r(\partial_{x}\omega)(x/r)\right|^{2}x^{2\varepsilon+n-2}$
and that, for $n\ge1$, the integrand is smaller than the integrable function $\left\Vert x\partial_{x}\omega\right\Vert _{L^{\infty}([0,\infty))}^{2}x^{2\varepsilon+n-2}$.
Moreover
$$
\lim_{r\to0}\left|\omega(x)x/r(\partial_{x}\omega)(x/r)\right|^{2}x^{2\varepsilon+n-2}=0.
$$
The result now follows again
by the dominated convergence theorem.\\
\emph{Step 3}: $\mathbb{C}_{\omega}\hookrightarrow H^{1}(\mathbb{B})$. It suffices to show that the constant function equal to one belongs to $H^{1}(\mathbb{B})$
by showing that $\lim_{\varepsilon\to0}\mathsf{x}^{\varepsilon}=1$ in $H^{1}(\mathbb{B})$. To this end, it is enough to show that
$$
\begin{aligned}\text{(i)} \quad & \lim_{\varepsilon\to0}\int_{0}^{1}\int_{\partial\mathbb{B}}\left|\omega(1-x^{\varepsilon})\right|^{2}x^{n}dxdy =0,\\
\text{(ii)}\quad & \int_{0}^{1}\int_{\partial\mathbb{B}}\left|\partial_{y_{j}}(\omega)-\partial_{y_{j}}(\omega x^{\varepsilon})\right|^{2}x^{n-2}dxdy =0,\\
\text{(iii)}\quad & \lim_{\varepsilon\to0}\int_{0}^{1}\int_{\partial\mathbb{B}}\left|x\partial_{x}(\omega)-x\partial_{x}(\omega x^{\varepsilon})\right|^{2}x^{n-2}dxdy =0.
\end{aligned}
$$
Again (i) follows directly from the dominated convergence theorem, (ii) is identically zero as the functions do not depend on $y$, and the integral in (iii) is smaller than two times
$$
\int_{0}^{1}\int_{\partial\mathbb{B}}\left|x\partial_{x}(\omega)-x^{\varepsilon}x\partial_{x}(\omega)\right|^{2}x^{n-2}dxdy+\int_{0}^{1}\int_{\partial\mathbb{B}}\left|\omega x\partial_{x}(x^{\varepsilon})\right|^{2}x^{n-2}dxdy.
$$
The first term can be dealt again by
dominated convergence. For the second one, note that
$$
\int_{0}^{1}\left|\omega x\partial_{x}(x^{\varepsilon})\right|^{2}x^{n-2}dx\le\frac{\varepsilon^{2}}{2\varepsilon+n-1},
$$
and the last term goes to zero, as $\varepsilon$ goes to zero.\\
\emph{Step 4}: If $\beta<1$, then $H^{1}(\mathbb{B})\hookrightarrow\mathcal{H}^{1,\beta}(\mathbb{B})$. By density, it is enough to show that there is a constant $C>0$ such
that $\left\Vert u\right\Vert _{\mathcal{H}^{1,\beta}(\mathbb{B})}\le C\left\Vert u\right\Vert _{H^{1}(\mathbb{B})}$,
for all $u\in C_{c}^{\infty}(\mathbb{B}^{\circ})$.\\
If $k+|\alpha|=1$, then, in local coordinates on $[0,1)\times\partial\mathcal{B}$,
we have
$$
\left|x^{\frac{n+1}{2}-\beta}(x\partial_{x})^{k}\partial_{y}^{\alpha}(\omega(x)u(x,y))\right|\le\left|x^{\frac{n-1}{2}}(x\partial_{x})^{k}\partial_{y}^{\alpha}(\omega(x)u(x,y))\right|.
$$
If $k+|\alpha|=0$, then as
$$
\omega(x)u(x,y)=-\int_{x}^{1}\frac{\partial}{\partial s}(\omega(s)u(s,y))ds,
$$
we have
\begin{eqnarray*}\lefteqn{\int_{\partial\mathbb{B}}\int_{0}^{1}\left|x^{\frac{n+1}{2}-\beta}\omega(x)u(x,y)\right|^{2}\frac{dx}{x}dy}\\
 & =&\int_{\partial\mathbb{B}}\int_{0}^{1}x^{n-2\beta}\left|\int_{x}^{1}s^{-\frac{n}{2}}s^{\frac{n}{2}}\frac{\partial}{\partial s}\left(\omega(s)u(s,y)\right)ds\right|^{2}dxdy\\
 &\le&\int_{\mathbb{B}}\int_{0}^{1}x^{n-2\beta}\left(\int_{x}^{1}s^{-n}ds\right)\left(\int_{x}^{1}s^{n}\left|\frac{\partial}{\partial s}(\omega(s)u(s,y))\right|^{2}ds\right)dxdy\\
 & = &\int_{0}^{1}x^{n-2\beta}\left(\int_{x}^{1}s^{-n}ds\right)\left(\int_{\partial\mathbb{B}}\int_{x}^{1}\left|s^{\frac{n-1}{2}}\left(s\frac{\partial}{\partial s}\right)\left(\omega(s)u(s,y)\right)\right|^{2}\frac{ds}{s}dy\right)dx\\
 & \le &\int_{0}^{1}x^{n-2\beta}\left(\int_{x}^{1}s^{-n}ds\right)dx\left\Vert u\right\Vert _{H^{1}(\mathbb{B})}^{2},
\end{eqnarray*}
where we have used Cauchy-Schwarz in the first inequality. The last integral is finite for $n\ge1$ and $\beta<1$.
\end{proof}
For functions in $H^{1}(\mathbb{B})$, we define
$$
(u,v)_{H_{0}^{1}(\mathbb{B})}=\int_{\mathbb{B}}\left\langle \nabla u,\nabla\overline{v}\right\rangle d\mu_{g}\,\,\,\text{and}\,\,\,\left\Vert u\right\Vert _{H_{0}^{1}}^2=\int_{\mathbb{B}}\left\langle \nabla u,\nabla\overline{u}\right\rangle d\mu_{g}.
$$
In particular,
\begin{equation}
(u,v)_{H^{1}(\mathbb{B})}:=(u,v)_{\mathcal{H}^{0,0}(\mathbb{B})}+(u,v)_{H_{0}^{1}(\mathbb{B})}\label{eq:NormWzeroandW}
\end{equation}
and
\begin{equation}
\left\Vert u\right\Vert _{H^{1}(\mathbb{B})}=\sqrt{\left\Vert u\right\Vert _{\mathcal{H}^{0,0}(\mathbb{B})}^{2}+\left\Vert u\right\Vert _{H_{0}^{1}(\mathbb{B})}^{2}}.\label{eq:equivH11}
\end{equation}
Moreover, whenever $u\in L^{1}(\mathbb{B})$, we define $(u)_{\mathbb{B}}:=\fint_{\mathbb{B}}udx=|\mathbb{B}|^{-1}\int_{\mathbb{B}}ud\mu_{g}$,
where $|\mathbb{B}|=\int_{\mathbb{B}}d\mu_{g}$ is the area of $\mathbb{B}$.
\begin{lem}[Poincar\'e-Wirtinger inequality]
\label{lem:Poincar=0000E9-Wirtinger}
There is a constant $C>0$ such that 
$$
\left\Vert u-(u)_{\mathbb{B}}\right\Vert _{\mathcal{H}^{0,0}(\mathbb{B})}\le C\left\Vert u\right\Vert _{H_{0}^{1}(\mathbb{B})},\quad \forall u\in H^{1}(\mathbb{B}).
$$
\end{lem}

\begin{proof}
The proof follows the same argument as in the proof of \cite[Theorem 1 of Section 5.8]{Evans}, using the fact that $H^{1}(\mathbb{B})$
is compactly embedded in $\mathcal{H}^{0,0}(\mathbb{B})$.
\end{proof}
\begin{defn}
Denote by $H_{0}^{1}(\mathbb{B})$ the space of all $u\in H^{1}(\mathbb{B})$
such that $(u)_{\mathbb{B}}=0$.
\end{defn}

It is clear that $H^{1}(\mathbb{B})=H_{0}^{1}(\mathbb{B})\oplus\mathbb{C}$,
where $\mathbb{C}$ is identified with the set of constant functions.
Moreover, applying Lemma \ref{lem:Poincar=0000E9-Wirtinger} with $(u)_\mathbb{B}=0$, we see that the map
$H_{0}^{1}(\mathbb{B})\ni u\mapsto\left\Vert u\right\Vert _{H_{0}^{1}(\mathbb{B})}\in\mathbb{R}$
is equivalent to the $H^{1}(\mathbb{B})$ norm.
\begin{defn}
\label{def:H-10} We denote by $H^{-1}(\mathbb{B})$ the dual space of $H^{1}(\mathbb{B})$
and by $H_{0}^{-1}(\mathbb{B})\subset H^{-1}(\mathbb{B})$ its subspace
defined by
$$
H_{0}^{-1}(\mathbb{B}):=\{ u\in H^{-1}(\mathbb{B}):\,\left\langle u,1\right\rangle _{H^{-1}(\mathbb{B})\times H^{1}(\mathbb{B})}=0\} .
$$
\end{defn}

Using the fact that $H^{1}(\mathbb{B})=H_{0}^{1}(\mathbb{B})\oplus\mathbb{C}$,
we can see that the map $H_{0}^{-1}(\mathbb{B})\ni u\mapsto\left.u\right|_{H_{0}^{1}(\mathbb{B})}\in\mathcal{L}(H_{0}^{1}(\mathbb{B}),\mathbb{C})$
is bijective, that is, $H_{0}^{-1}(\mathbb{B})$ can be identified
with the dual of $H_{0}^{1}(\mathbb{B})$.

\begin{prop}
\label{prop:HzerobetadualH1}Let $u\in\mathcal{H}^{0,\beta}(\mathbb{B})$,
for some $\beta>-1$. Then $T_{u}:H^{1}(\mathbb{B})\to\mathbb{C}$ and $\widetilde{T}_{u}:H^{1}(\mathbb{B})\to\mathbb{C}$
defined by 
$$
\begin{aligned}T_{u}(v) & =\int_{\mathbb{B}}uvd\mu_{g},\\
\widetilde{T}_{u}(v) & =\int_{\mathbb{B}}(u-(u)_{\mathbb{B}})vd\mu_{g}
\end{aligned}
$$
are continuous. Moreover the functional $\widetilde{T}_{u}$ belongs
to $H_{0}^{-1}(\mathbb{B})$ and $\widetilde{T}_{u}|_{H_{0}^{1}(\mathbb{B})}=\left.T_{u}\right|_{H_{0}^{1}(\mathbb{B})}$.
\end{prop}

\begin{proof}
Since $\beta>-1$, we have the inclusion $\mathcal{H}^{0,\beta}(\mathbb{B})\hookrightarrow L^{1}(\mathbb{B})$.
In fact,
$$
\begin{aligned}\int_{\mathbb{B}}\left|u\right|d\mu_{g} & =\int_{\mathbb{B}}\mathsf{x}^{-\beta}\left|u\right|\mathsf{x}^{\beta}d\mu_{g}\le\left(\int_{\mathbb{B}}\mathsf{x}^{2\beta}d\mu_{g}\right)^{1/2}\left\Vert u\right\Vert _{\mathcal{H}^{0,\beta}(\mathbb{B})},\end{aligned}
$$
 due to Remark \ref{rem:xmellinL2weight}. Note that $\int_{\mathbb{B}}\mathsf{x}^{2\beta}d\mu_{g}$
is finite, as $\int_{0}^{1}x^{n+2\beta}dx<\infty$. The fact that
$\mathcal{H}^{0,\beta}(\mathbb{B})\hookrightarrow L^{1}(\mathbb{B})$ ensures
that $(u)_{\mathbb{B}}$ is well defined.

In order to prove that $T_{u}$ is continuous, let us assume, without
loss of generality, that $-1<\beta\le0$. We denote by $\mathcal{I}_{-\beta}:H^{1}(\mathbb{B})\to\mathcal{H}^{0,-\beta}(\mathbb{B})$
the continuous inclusion from Proposition \ref{prop:inclusionH1mellin}.
Then, we have that 
$$
T_{u}(v)=\int_{\mathbb{B}}uvd\mu_{g}=\left\langle u,\mathcal{I}_{-\beta}(v)\right\rangle _{\mathcal{H}^{0,\beta}(\mathbb{B})\times\mathcal{H}^{0,-\beta}(\mathbb{B})}.
$$

Therefore $T_{u}$ is continuous as it is the composition of continuous functions.
The continuity of $\widetilde{T}_{u}$ follows similarly. The fact
that $\widetilde{T}_{u}|_{H_{0}^{1}(\mathbb{B})}=\left.T_{u}\right|_{H_{0}^{1}(\mathbb{B})}$
follows from the fact that the integral of $v$ is equal to zero if $v\in H^{1}_{0}(\mathbb{B})$.
\end{proof}
A version of Gauss theorem can also be proved for $H^{1}(\mathbb{B})$. It simplifies and improves \cite[Lemma 4.3]{LopesRoidos}.
\begin{thm}
[Gauss theorem] \label{thm:Gauss-Theorem} Let $u$ and $v$ belong
to $H^{1}(\mathbb{B})$ and $\Delta v\in\mathcal{H}^{0,\gamma}(\mathbb{B})$,
for some $\gamma>-1$. Then 
$$
\int_{\mathbb{B}}\left\langle \nabla u,\nabla v\right\rangle d\mu_{g}=-\int_{\mathbb{B}}u\Delta vd\mu_{g}.
$$
In particular, if
$u\in H^{1}(\mathbb{B})$ is such that $\Delta u\in\mathcal{H}^{0,\gamma}(\mathbb{B})$,
for some $\gamma>-1$, then $\int_{\mathbb{B}}\Delta u\,d\mu_{g}=0.$
\end{thm}

\begin{proof}
Without loss of generality, we assume that $-1<\gamma\le0$. First
we note that for $v$ and $u$ in $C_{c}^{\infty}(\mathbb{B}^{\circ})$,
we have
\begin{equation}
\int_{\mathbb{B}}\left\langle \nabla u,\nabla v\right\rangle d\mu_{g}=-\int_{\mathbb{B}}u\Delta vd\mu_{g}=-\left\langle \Delta v,u\right\rangle _{\mathcal{D}'(\mathbb{B}^{\circ})\times C_{c}^{\infty}(\mathbb{B}^{\circ})},\label{eq:forthegreen}
\end{equation}
where $\mathcal{D}'(\mathbb{B}^{\circ})$ stands for the dual space
of $C_{c}^{\infty}(\mathbb{B}^{\circ})$.

For $v\in H^{1}(\mathbb{B})$, we can choose a sequence of functions
in $C_{c}^{\infty}(\mathbb{B}^{\circ})$ that converge to $v$ in
$H^{1}(\mathbb{B})$ and, therefore, also in $\mathcal{D}'(\mathbb{B}^{\circ})$.
Hence the equality between the first and the third term of
\eqref{eq:forthegreen} still holds for all $v\in H^{1}(\mathbb{B})$
and $u\in C_{c}^{\infty}(\mathbb{B}^{\circ})$.

Moreover if $v\in H^{1}(\mathbb{B})$ and $\Delta v\in\mathcal{H}^{0,\gamma}(\mathbb{B})\subset L_{\text{loc}}^{2}(\mathbb{B}^{\circ})\subset L_{\text{loc}}^{1}(\mathbb{B}^{\circ})$,
we have again 
\begin{equation}
\left\langle \Delta v,u\right\rangle _{\mathcal{D}'(\mathbb{B}^{\circ})\times C_{c}^{\infty}(\mathbb{B}^{\circ})}=\int_{\mathbb{B}}u\Delta vd\mu_{g},
\label{eq:forthegreen2}\end{equation}
for all $u\in C_{c}^{\infty}(\mathbb{B}^{\circ})$.

Finally, if $v\in H^{1}(\mathbb{B})$, $\Delta v\in\mathcal{H}^{0,\gamma}(\mathbb{B})$
and $u\in H^{1}(\mathbb{B})$, we take a sequence in $C_{c}^{\infty}(\mathbb{B}^{\circ})$
that converges to $u$ in $H^{1}(\mathbb{B})$. As $H^{1}(\mathbb{B})\hookrightarrow\mathcal{H}^{0,-\gamma}(\mathbb{B})$,
the sequence will also converge to $u$ in $\mathcal{H}^{0,-\gamma}(\mathbb{B})$.
Using the duality of $\mathcal{H}^{0,-\gamma}(\mathbb{B})$ and $\mathcal{H}^{0,\gamma}(\mathbb{B})$, \eqref{eq:forthegreen} and \eqref{eq:forthegreen2},
we obtain our result.
\end{proof}
\begin{prop}
\label{prop:deltaisomorphism}The operator $\Delta:H_{0}^{1}(\mathbb{B})\to H_{0}^{-1}(\mathbb{B})$
is well-defined, continuous and bijective.
\end{prop}

Note that $\Delta u\in\mathcal{D}'(\mathbb{B}^{\circ})$ is always
well defined. As $C_{c}^{\infty}(\mathbb{B}^{\circ})$ is dense in
$H^{1}(\mathbb{B})$, $H_{0}^{-1}(\mathbb{B})$ can be easily identified,
by restricting to $C_{c}^{\infty}(\mathbb{B}^{\circ})$, with a subset
of $\mathcal{D}'(\mathbb{B}^{\circ})$.
\begin{proof}
By Riesz theorem, there is a bijective map $\mathcal{R}:H_{0}^{1}(\mathbb{B})\to H_{0}^{-1}(\mathbb{B})$
such that
$$
\left\langle \mathcal{R}u,\overline{v}\right\rangle _{H_{0}^{-1}(\mathbb{B})\times H_{0}^{1}(\mathbb{B})}=(u,v)_{H_{0}^{1}(\mathbb{B})}=\int_{\mathbb{B}}\left\langle \nabla u,\nabla\overline{v}\right\rangle d\mu_{g}.
$$
Then, the operator $\Delta$ can be easily identified with $-\mathcal{R}$.
\end{proof}
\begin{cor}
\label{lem:EquivalenceofnormDeltauandu-1} There is a constant $C>0$
such that
$$
\left\Vert u\right\Vert _{H_{0}^{-1}(\mathbb{B})}\le C\left\Vert \Delta u\right\Vert _{H_{0}^{-1}(\mathbb{B})},\quad \forall u\in H_{0}^{1}(\mathbb{B}).
$$
\end{cor}

\begin{proof}
This follows from the continuous inclusion $H_{0}^{1}(\mathbb{B})\hookrightarrow H_{0}^{-1}(\mathbb{B})$, provided by Propositions 
\ref{prop:inclusionH1mellin} and \ref{prop:HzerobetadualH1},
and Proposition \ref{prop:deltaisomorphism}.
\end{proof}

\subsection{Sobolev immersions}

In this section, we prove some embeddings concerning Mellin-Sobolev
spaces.
\begin{prop}
Suppose that $p\in[2,\infty)$, if $n=1$, and $p\in[2,(2n+2)/(n-1)]$,
if $n\ge2$. Then, for each $\gamma\in\mathbb{R}$, we have $\mathcal{H}^{1,\gamma}(\mathbb{B})\hookrightarrow\mathsf{x}^{\gamma-(n+1)(1/2-1/p)}L^{p}(\mathbb{B})$.
In particular, if $\gamma\ge(n+1)(1/2-1/p)$, then
$\mathcal{H}^{1,\gamma}(\mathbb{B})\hookrightarrow L^{p}(\mathbb{B})$.
\end{prop}

\begin{proof}
First, we note that
$$
\mathcal{H}^{1,\gamma}(\mathbb{B})\subset H_{\text{loc}}^{1}(\mathbb{B}^{\circ})\subset L_{\text{loc}}^{p}(\mathbb{B}^{\circ}).
$$
Therefore, it is enough to understand the behavior of elements of
$\mathcal{H}^{1,\gamma}(\mathbb{B})$ in the neighborhood of the conical
tip. We have $H^{1}(\mathbb{R}^{n+1})\hookrightarrow L^{p}(\mathbb{R}^{n+1})$.
Hence
\begin{eqnarray}\nonumber
\lefteqn{ \left\Vert (x,y)\mapsto e^{(\gamma-\frac{n+1}{2})x}\omega(e^{-x})\phi_{j}(y)u(e^{-x},y)\right\Vert _{L^{p}(\mathbb{R}^{n+1})}}\\\label{eq:Im1}
 && \le C\left\Vert (x,y)\mapsto e^{(\gamma-\frac{n+1}{2})x}\omega(e^{-x})\phi_{j}(y)u(e^{-x},y)\right\Vert _{H^{1}(\mathbb{R}^{n+1})}\le C\left\Vert u\right\Vert _{\mathcal{H}^{1,\gamma}(\mathbb{B})},
\end{eqnarray}
where we recall that $\{\phi_{j}\}_{j\in J}$ is a partition of unity
of $\partial\mathcal{B}$ and $\omega$ is as Definition \ref{def:mellinsobolevspaces}. A change of variables $e^{-x}\mapsto x$ in \eqref{eq:Im1} implies
$$
\left\Vert \mathsf{x}^{(n+1)(\frac{1}{2}-\frac{1}{p})-\gamma}u\right\Vert _{L^{p}(\mathbb{B})}\le C\left\Vert u\right\Vert _{\mathcal{H}^{1,\gamma}(\mathbb{B})},
$$
which shows the claim.
\end{proof}
\textcolor{green}{}%
\begin{comment}
\textcolor{green}{Note that if $n\ge2$ and $p=\frac{2(n+1)}{n-1}$,
then $\mathcal{H}^{1,\beta}(\mathbb{B})\hookrightarrow L^{p}(\mathbb{B})$
if 
$$
\begin{aligned}\beta & \ge(n+1)(\frac{1}{2}-\frac{1}{p})=\frac{(n+1)}{2}(1-\frac{n-1}{(n+1)})\\
 & =\frac{(n+1)}{2}(\frac{n+1}{n+1}+\frac{1-n}{(n+1)})=\frac{(n+1)}{2}(\frac{n+1+1-n}{n+1})\\
 & =\frac{(n+1)}{2}(\frac{2}{n+1})=1.
\end{aligned}
$$
}

\textcolor{green}{Very probably $H^{1}(\mathbb{B})\subset L^{6}(\mathbb{B})$?}
\end{comment}

\begin{cor}
\label{cor:Mellinsobolevembedding}The following continuous inclusions
hold:\\
\emph{(i)} If $\dim(\mathbb{B})\in\{2,3\}$, then $H^{1}(\mathbb{B})\hookrightarrow L^{4}(\mathbb{B})$.\\
\emph{(ii)} If $\dim(\mathbb{B})=2$, then $H^{1}(\mathbb{B})\hookrightarrow L^{6}(\mathbb{B})$.\\
\emph{(iii)} If $\dim(\mathbb{B})=3$ and $\alpha>0$, then $H^{1}(\mathbb{B})\hookrightarrow\mathsf{x}^{-\alpha}L^{6}(\mathbb{B})$.
\end{cor}

\begin{proof}
We have seen that $H^{1}(\mathbb{B})\hookrightarrow\mathcal{H}^{1,\beta}(\mathbb{B})$,
for all $\beta<1$. For $\dim(\mathbb{B})=2$, if we choose $1/2\le\beta<1$,
we have $\mathcal{H}^{1,\beta}(\mathbb{B})\hookrightarrow L^{4}(\mathbb{B})$.
For $\dim(\mathbb{B})=3$, if we choose $3/4\le\beta<1$, we have
$\mathcal{H}^{1,\beta}(\mathbb{B})\hookrightarrow L^{4}(\mathbb{B})$.
For $\dim(\mathbb{B})=2$, if we choose if $2/3\le\beta<1$, then
$\mathcal{H}^{1,\beta}(\mathbb{B})\hookrightarrow L^{6}(\mathbb{B})$.
Finally, for $\dim(\mathbb{B})=3$ and $\alpha>0$, then choose $1-\alpha\le\beta<1$
so that $\mathcal{H}^{1,\beta}(\mathbb{B})\hookrightarrow\mathsf{x}^{-\alpha}L^{6}(\mathbb{B})$.
\end{proof}
\textcolor{green}{}%

\section{Realizations of the Laplacian and bi-Laplacian\label{sec:Realizations}}

We start with some basic concepts on analytic semigroup theory.
\begin{defn}
Let $X$ be a complex Banach space and $A:\mathcal{D}(A)\to X$ be
a densely defined closed operator in $X$. We say that $A$ is a
\emph{negative generator of an analytic semigroup} if for some $\delta,C>0$
we have
$$
\{\lambda\in\mathbb{C}:\,\text{Re}(\lambda)>-\delta\}\subset\rho(A)\quad\text{and}\quad\left\Vert (\lambda-A)^{-1}\right\Vert _{\mathcal{L}(X)}\le C/|\lambda|,\,\,\,\text{Re}(\lambda)>-\delta.
$$
\end{defn}

The semigroup associated to a negative generator $A$ is denoted
$e^{tA}\in\mathcal{L}(X)$, see e.g. \cite[Chapter I.1.2]{Am}. Both semigroup and the complex powers
$(-A)^{z}:\mathcal{D}((-A)^{z})\to X$, $z\in\mathbb{C}$, can be
defined by Cauchy's integral formula, see e.g. \cite[Theorem III.4.6.5]{Am}. In the case of $(-A)^{it}\in\mathcal{L}(X)$ for all $t\in\mathbb{R}$ and $\left\Vert (-A)^{it}\right\Vert \le Me^{\phi|t|}$, for some $M>0$ and $\phi\ge 0$,
we say that $-A$ has \emph{bounded imaginary powers} and denote by
$-A\in\mathcal{BIP}(\phi)$, see e.g. \cite[Chapter III.4.7]{Am}. Recall that if $-A\in\mathcal{BIP}$, then $[X,\mathcal{D}(A)]_{\theta}=\mathcal{D}((-A)^{\theta})$, see e.g. \cite[Theorem 4.17.]{Lunardi}.

Next we recall some basic facts from the {\em cone calculus}, for more details we refer to \cite{CSS1,GilMendoza,Le,GKM,SS1,SS,Schu}. The Laplacian $\Delta$, as a cone differential operator, acts naturally on scales of Mellin-Sobolev spaces. Let us consider it as an unbounded operator
in $\mathcal{H}^{s,\gamma}(\mathbb{B})$, $s,\gamma\in\mathbb{R}$,
with domain $C_{c}^{\infty}(\mathbb{B}^{\circ})$. Denote by $\Delta_{s,\min}$
the minimal extension (i.e. the closure) of $\Delta$ and by $\Delta_{s,\max}$
the maximal extension, defined as usual by
$$
\mathcal{D}(\Delta_{s,\max})=\Big\{ u\in\mathcal{H}^{s,\gamma}(\mathbb{B})\,|\,\Delta u\in\mathcal{H}^{s,\gamma}(\mathbb{B})\Big\}.
$$

An important result in the field of cone differential operators tells
us that those two domains differ in general, unlike the case of closed
manifolds. In particular, there exists an $s$-independent finite-dimensional
space $\mathcal{E}_{\max,\Delta,\gamma}$, that is called \emph{asymptotics
space}, such that
\begin{equation}\label{maxdom}
\mathcal{D}(\Delta_{s,\max})=\mathcal{D}(\Delta_{s,\min})\oplus\mathcal{E}_{\max,\Delta,\gamma}.
\end{equation}
More precisely, $\mathcal{E}_{\max,\Delta,\gamma}$ consists of linear
combinations of smooth functions in $\mathbb{B}^{\circ}$. Those functions
vanish outside the collar neighborhood and, in local coordinates $(x,y)\in[0,1)\times\partial\mathcal{B}$,
can be written as $\omega(x)c(y)x^{-\rho}\log^{k}(x)$. The function
$\omega$ is the cut-off function defined on Section \ref{sec:Function-spaces},
$c\in C^{\infty}(\partial\mathbb{B})$, $\rho\in\{z\in\mathbb{C}\,|\,\mathrm{Re}(z)\in[\frac{n-3}{2}-\gamma,\frac{n+1}{2}-\gamma)\}$
and $k\in\{0,1\}$. Here, the metric $h(\cdot)$ determines explicitly
the exponents $\rho$. As for the minimal domain, it can be proved that $\mathcal{D}(\Delta_{s,\min})=\mathcal{H}^{s+2,\gamma+2}(\mathbb{B})$,
whenever
$$
\frac{n-3}{2}-\gamma\notin\Big\{\frac{n-1}{2}\pm\sqrt{\Big(\frac{n-1}{2}\Big)^{2}-\lambda_{j}}:\,j\in\mathbb{N}\Big\}.
$$

A suitable choice of the domain of the Laplacian is given below. For
this, we denote by $\mathbb{C}_{\omega}$ the finite dimensional
space of functions that are equal to zero outside the collar neighborhood and that close to the singularities are expressed by $\sum_{j=1}^{N}c_{j}\omega_{j}$, where $N$ is the
number of connected components of $\partial\mathcal{B}$ and $\omega_{j}$
are the restrictions of $\omega$ to these components. Also denote by $\mathbb{R}_\omega$ the subspace of $\mathbb{C}_\omega$ such that $c_j\in\mathbb{R}$.
\begin{thm}
\label{thm:Laplacian} \cite[Theorem 6.7]{SS1} Let
\begin{equation}\label{gammachoice}
\frac{n-3}{2}<\gamma<\min\Big\{-1+\sqrt{\Big(\frac{n-1}{2}\Big)^{2}-\lambda_{1}},\frac{n+1}{2}\Big\},
\end{equation}
where $\lambda_{1}$ is the greatest non-zero eigenvalue of the boundary
Laplacian $\Delta_{h(0)}$ on $(\partial\mathcal{B},h(0))$. Then
for every $c,\phi>0$, the operator $\Delta-c:\mathcal{H}^{s+2,\gamma+2}(\mathbb{B})\oplus\mathbb{C}_{\omega}\to\mathcal{H}^{s,\gamma}(\mathbb{B})$
is a negative generator of an analytic semigroup such that $-\Delta+c\in\mathcal{BIP}(\phi)$. The Laplacian $\Delta:\mathcal{H}^{s+2,\gamma+2}(\mathbb{B})\oplus\mathbb{C}_{\omega}\to\mathcal{H}^{s,\gamma}(\mathbb{B})$ will be denoted by $\Delta_s$.
\end{thm}

Similarly to the Laplacian, the bi-Laplacian has a minimal extension $\Delta_{s,\min}^{2}$, which is the closure of $\Delta^{2}:C_{c}^{\infty}(\mathbb{B}^{\circ})\to\mathcal{H}^{s,\gamma}(\mathbb{B})$ satisfying 
$$
\mathcal{H}^{s+4,\gamma+4}(\mathbb{B})\hookrightarrow\mathcal{D}(\Delta_{s,\min}^{2})\hookrightarrow\cap_{\varepsilon>0}\mathcal{H}^{s+4,\gamma+4-\varepsilon}(\mathbb{B}),
$$ 
and a maximal extension $\Delta_{s,\max}^{2}$, whose domain is $\{u\in\mathcal{H}^{s,\gamma}(\mathbb{B}):\Delta^{2}u\in\mathcal{H}^{s,\gamma}(\mathbb{B})\}$. They are related by $\mathcal{D}(\Delta_{s,\max}^{2})=\mathcal{D}(\Delta_{s,\min}^{2})\oplus\mathcal{E}_{\max,\Delta^{2},\gamma}$,where $\mathcal{E}_{\max,\Delta^{2},\gamma}$ is a finite dimensional space consisting of linear combinations of functions of the form $\omega(x)c(y)x^{-\rho}\log^{k}(x)$, with $k\in\{0,1,2,3\}$ and $\rho\in\{z\in\mathbb{C}:\text{Re}(z)\in[\frac{n-7}{2}-\gamma,\frac{n+1}{2}-\gamma)\}$.

We will choose a bi-Laplacian domain $\mathcal{D}(\Delta_{s}^{2})$ satisfying 
$$
\mathcal{D}(\Delta_{s,\min}^{2})\hookrightarrow\mathcal{D}(\Delta_{s}^{2})\hookrightarrow\mathcal{D}(\Delta_{s,\max}^{2}).
$$ 
The definition and the properties of the operator $\Delta_{s}^{2}:\mathcal{D}(\Delta_{s}^{2})\to\mathcal{H}^{s,\gamma}(\mathbb{B})$ are explained in the following corollary of Theorem \ref{thm:Laplacian}.

\begin{cor}
[bi-Laplacian] \label{cor:(Bi-Laplacian)} Consider the bi-Laplacian
$\mathcal{D}(\Delta_{s}^{2})=\{u\in\mathcal{D}(\Delta_{s}):\Delta_{s}u\in\mathcal{D}(\Delta_{s})\}$,
where $\Delta_{s}$ is as in Theorem \ref{thm:Laplacian}. Then, there exists an $s$-independent finite dimensional space $\mathcal{E}_{\Delta^{2},\gamma}$
contained in $\ensuremath{\mathcal{H}^{\infty,\gamma+2+\alpha_{0}}(\mathbb{B})}\cap\mathcal{E}_{\max,\Delta^{2},\gamma}$,
for some $\alpha_{0}>0$, such that \eqref{bilapintro1} holds. In particular, $\mathcal{D}(\Delta_{s}^{2})\hookrightarrow\mathcal{H}^{s+4.\gamma+4-\varepsilon}\oplus\mathbb{C}_{\omega}\oplus\mathcal{E}_{\Delta^{2},\gamma}$,
for all $\varepsilon>0$. The space $\mathcal{E}_{\Delta^{2},\gamma}$
consists of $C^{\infty}(\mathbb{B}^{\circ})$-functions, which in
local coordinates $(x,y)\in[0,1)\times\partial\mathcal{B}$, are of
the form $\omega(x)c(y)x^{\rho}\ln^{k}(x)$, where $c\in C^{\infty}(\partial\mathbb{B})$,
$\rho\in\{z\in\mathbb{C} : \mathrm{Re}(z)\in[\frac{n-7}{2}-\gamma,\frac{n-3}{2}-\gamma)\}$
and $k\in\{0,1,2,3\}$; for more details we refer to \cite[Section 3.2]{LopesRoidos}.

Moreover, for every $\phi>0$, the operator $A_{s}:=-(1-\Delta_{s})^{2}:\mathcal{D}(\Delta_{s}^{2})\to\mathcal{H}^{s,\gamma}(\mathbb{B})$
is a negative generator of an analytic semigroup such that $(1-\Delta_{s})^{2}\in\mathcal{BIP}(\phi)$, see \cite[Proposition 3.6]{LopesRoidos}.
\end{cor}

For $\alpha\in[0,2]$, we define $X_{\alpha}^{s}:=\mathcal{D}((-A_{s})^{\alpha/2})=[\mathcal{H}^{s,\gamma}(\mathbb{B}),\mathcal{D}(\Delta_{s}^{2})]_{\alpha/2}$. Notice that $X_{2}^{s}=\mathcal{D}(\Delta_{s}^{2})$ and $X_{1}^{s}=\mathcal{D}(\Delta_{s})$. We also denote $X_{1}^{\infty}:=\cap_{s\ge0}X_{1}^{s}$ and $X_{2}^{\infty}:=\cap_{s\ge0}X_{2}^{s}$. We close this section with some facts about the complex interpolation of the domain of the Laplacian and bi-Laplacian. 

\begin{prop}\label{thm:Interpolation} 
Let $\alpha\in(0,1)$ and assume that $\gamma$ satisfies \eqref{gammachoice}.\\
\emph{(i)} If $\alpha\notin\{\frac{1-\gamma}{2}\pm\frac{1}{2}\sqrt{(\frac{n-1}{2})^{2}-\lambda_{j}}:j\in\mathbb{N}\}$,
then $[\mathcal{H}^{s,\gamma}(\mathbb{B}),X_{1}^{s}]_{\alpha}=\mathcal{H}^{s+2\alpha,\gamma+2\alpha}(\mathbb{B})\oplus\mathbb{C}_{\omega}$.\\
\emph{(ii)} We always have $\mathcal{H}^{s+4\alpha,\gamma+4\alpha}(\mathbb{B})\oplus\mathbb{C}_{\omega}\oplus\mathcal{E}_{\Delta^{2},\gamma}\subset[\mathcal{H}^{s,\gamma}(\mathbb{B}),X_{2}^{s}]_{\alpha}$.\\
\emph{(iii)} If $2\alpha\notin\{\frac{1-\gamma}{2}\pm\frac{1}{2}\sqrt{(\frac{n-1}{2})^{2}-\lambda_{j}}:j\in\mathbb{N}\}$,
then 
$$
[\mathcal{H}^{s,\gamma}(\mathbb{B}),X_{2}^{s}]_{\alpha}\hookrightarrow\cap_{\varepsilon>0}\mathcal{H}^{s+4\alpha,\gamma+4\alpha-\varepsilon}\oplus\mathbb{C}_{\omega}\oplus \underline{\mathcal{E}}_{\Delta^{2},\gamma+4(\alpha-1)},
$$
where $\underline{\mathcal{E}}_{\Delta^{2},\gamma+4(\alpha-1)}\subset\mathcal{H}^{\infty,\gamma+2}(\mathbb{B})$ is a subspace of the asymptotic space $\mathcal{E}_{\max,\Delta^{2},\gamma+4(\alpha-1)}$. We note that the sums above are not necessarily direct.
\end{prop}
\begin{proof}
(i) It follows from \cite[Lemma 4.5 (i)]{LopesRoidos}.\\
(ii) It follows from the inclusions $\mathbb{C}_{\omega}\oplus\mathcal{E}_{\Delta^{2},\gamma}\hookrightarrow\mathcal{D}(\Delta_{s}^{2})=X_{2}^{s}$
and
$$
\mathcal{H}^{s+4\alpha,\gamma+4\alpha}(\mathbb{B}) =[\mathcal{H}^{s,\gamma}(\mathbb{B}),\mathcal{H}^{s+4,\gamma+4}(\mathbb{B})]_{\alpha}\hookrightarrow[\mathcal{H}^{s,\gamma}(\mathbb{B}),X_{2}^{s}]_{\alpha},
$$
where in the first equality we have used \cite[Lemma 3.3 (iii)]{LopesRoidos}.\\
(iii) \emph{Case $\alpha\in(0,1/2]$.} We have
$$
[\mathcal{H}^{s,\gamma}(\mathbb{B}),X_{2}^{s}]_{\alpha}=\mathcal{D}((-A_{s})^{\alpha})=\mathcal{D}((1-\Delta_{s})^{2\alpha})=[\mathcal{H}^{s,\gamma}(\mathbb{B}),\mathcal{D}(\Delta_{s})]_{2\alpha}=\mathcal{H}^{s+4\alpha,\gamma+4\alpha}(\mathbb{B})\oplus\mathbb{C}_{\omega}.
$$
In the first and third equalities we have used the $\mathcal{BIP}$
property of $1-\Delta_s$ and $-A_s$. In the last equality, we
have used (i).\\
\emph{Case $\alpha\in(1/2,1)$.} We first note that
$$
\Delta^{2}[\mathcal{H}^{s,\gamma}(\mathbb{B}),\mathcal{D}(\Delta_{s}^{2})]_{\alpha}\hookrightarrow[\mathcal{H}^{s-4,\gamma-4}(\mathbb{B}),\mathcal{H}^{s,\gamma}(\mathbb{B})]_{\alpha}=\mathcal{H}^{s+4(\alpha-1),\gamma+4(\alpha-1)}(\mathbb{B}),
$$
where \cite[Lemma 3.3 (iii)]{LopesRoidos} was used in the last equality. Hence, we have
\begin{equation}
[\mathcal{H}^{s,\gamma}(\mathbb{B}),X_{2}^{s}]_{\alpha}\hookrightarrow \mathcal{D}(\Delta_{s+4(\alpha-1),\max}^{2})\subset\mathcal{H}^{s+4\alpha,\gamma+4\alpha-\varepsilon}(\mathbb{B})\oplus\mathcal{E}_{\max,\Delta^{2},\gamma+4(\alpha-1)}.\label{eq:1}
\end{equation}
Moreover,
\begin{equation}
[\mathcal{H}^{s,\gamma}(\mathbb{B}),X_{2}^{s}]_{\alpha}=\mathcal{D}((1-\Delta_{s})^{2\alpha})\hookrightarrow\mathcal{D}(\Delta_{s})=\mathcal{H}^{s+2,\gamma+2}(\mathbb{B})\oplus\mathbb{C}_{\omega},\label{eq:2}
\end{equation}
where in the first equality we have used the $\mathcal{BIP}$ property of $-A_s$. Now let $u\in[\mathcal{H}^{s,\gamma}(\mathbb{B}),X_{2}^{s}]_{\alpha}$.
By (\ref{eq:1}), we have that $u=v + w$, where $v\in\mathcal{H}^{s+4\alpha,\gamma+4\alpha-\varepsilon}(\mathbb{B})$
and $w\in\mathcal{E}_{\max,\Delta^{2},\gamma+4(\alpha-1)}$. Hence, by (\ref{eq:2}) and for sufficiently small $\varepsilon >0$, we
have
$$
w=u-v\in\mathcal{H}^{s+2,\gamma+2}(\mathbb{B})\oplus\mathbb{C}_{\omega}+\mathcal{H}^{s+4\alpha,\gamma+4\alpha-\varepsilon}(\mathbb{B})=\mathcal{H}^{s+2,\gamma+2}(\mathbb{B})\oplus\mathbb{C}_{\omega}.
$$
Therefore 
$$
w\in\left(\mathcal{H}^{s+2,\gamma+2}(\mathbb{B})\oplus\mathbb{C}_{\omega}\right)\cap\mathcal{E}_{\max,\Delta^{2},\gamma+4(\alpha-1)}\subset\mathbb{C}_{\omega}\oplus\left(\mathcal{E}_{\max,\Delta^{2},\gamma+4(\alpha-1)}\cap\mathcal{H}^{s+2,\gamma+2}(\mathbb{B})\right),
$$
which concludes the proof.
\end{proof}

\section{\label{sec:Existence-and-regularity-1}Existence
and regularity of the global attractors}

For the rest of the paper $\gamma$ is fixed and satisfies \eqref{gamma}. The constants $C>0$ may change along the computations.

In this section we prove part (i) of Theorem \ref{thm:MainTheorem}. In the sequel all the spaces we use are the real parts of the ones defined previously.
Recall that a global attractor for a semiflow
$T:[0,\infty)\times X\to X$ defined on a Hilbert $X$ is a compact
set $\mathcal{A}\subset X$ such that $T(t)\mathcal{A}:=\{T(t)x:x\in\mathcal{A}\}=\mathcal{A}$,
for all $t\ge0$, which, moreover, attracts all bounded sets $B\subset X$
in the following sense:
$$
\lim_{t\to\infty}(\sup_{b\in B}\inf_{a\in\mathcal{A}}\left\Vert T(t)b-a\right\Vert _{X})=0.
$$
If it exists, then it is unique.

As mentioned in the introduction, for any $s\ge0$, we can define a
semiflow $T$ in $X_{1}^{s}$. It is convenient, however, to restrict
$T$ to a smaller space. For this reason, we first prove the following proposition.
\begin{prop}
\label{prop:avaragedoesnotchange}Let $u_{0}\in X_{1}^{s}$, then
the function $(u)_{\mathbb{B}}:(0,\infty)\to\mathbb{R}$, defined
by
$$
(u)_{\mathbb{B}}(t):=\frac{1}{|\mathbb{B}|}\int_{\mathbb{B}}T(t)u_{0}d\mu_{g},\,\,t>0,
$$
is constant. Here $|\mathbb{B}|=\int_{\mathbb{B}}d\mu_{g}$.
\end{prop}

\begin{proof}
The proof is similar to \cite[Equation (4.61)]{Temam}, where we have
to take into account \eqref{extrareg} and Theorem \ref{thm:Gauss-Theorem}.
\end{proof}
Consider the Hilbert space
$$
X_{1,0}^{s}=\{u\in\mathcal{H}^{s+2,\gamma+2}(\mathbb{B})\oplus\mathbb{R}_{\omega}:\,(u)_{\mathbb{B}}=0\}.
$$
By Proposition \ref{prop:avaragedoesnotchange}, $T(t)X_{1,0}^{s}\subset X_{1,0}^{s}$.
Therefore, $T$ restricts to a semiflow on $X_{1,0}^{s}$. Concerning
the existence of global attractors, we recall the following result.

For two Hilbert spaces $X$ and $Y$ such that $Y\hookrightarrow X$
and a semiflow $T:[0,\infty)\times X\to X$, we define the $\omega$-limit
set $\omega_{Z}(B)$ of $B\subset X$, where $Z=X$ or $Y$, by
\begin{equation}
\omega_{Z}(B)=\Big\{ z\in Z\, :\, \exists\quad t_{n}\to\infty\quad\text{and}\quad\{x_{n}\}_{n\in\mathbb{N}}\subset B\quad\text{such that}\quad\lim_{n\to\infty}\left\Vert T(t_{n})x_{n}-z\right\Vert _{Z}=0\Big\} .\label{eq:omegalimit}
\end{equation}
If $Z=Y$, then the above definition requires that $T(t_{n})x_{n}\in Y$
for all $n\in\mathbb{N}$. In order to show existence and regularity
of global attractors, we prove the following variation of
\cite[Theorem 10.5]{Robinson}.
\begin{thm}
\label{thm:Globalattractor}Let $Y\hookrightarrow X$ be Hilbert spaces,
$T:[0,\infty)\times X\to X$ be a semiflow and $\mathcal{K}\subset Y$
be a compact set in $Y$. Assume that for all bounded sets $B\subset X$
there exists a constant $t_{B}>0$ such that, if $t>t_{B}$, then
$T(t)B\subset\mathcal{K}$. Under these conditions there exists a
(unique) global connected attractor $\mathcal{A}$ for the semiflows
$T$. Moreover, $\mathcal{A}=\omega_{X}(\mathcal{K})=\omega_{Y}(\mathcal{K})$
is contained in $Y$ and attracts bounded sets of $X$ in $Y$ in
the following sense: for any bounded set $B\subset X$, the set $T(t)B$
is bounded in $Y$ for large $t$ and 
$$
\lim_{t\to\infty}\sup_{b\in B}\inf_{a\in\mathcal{A}}\left\Vert T(t)b-a\right\Vert _{Y}=0.
$$
\end{thm}

\begin{proof}
First we show that $\omega_{X}(\mathcal{K})=\omega_{Y}(\mathcal{K})$.
Since $Y\hookrightarrow X$, the definition given by (\ref{eq:omegalimit})
implies that $\omega_{Y}(\mathcal{K})\subset\omega_{X}(\mathcal{K})$.
On the other hand, if $x\in\omega_{X}(\mathcal{K})$ and 
$$
\lim_{n\to\infty}\left\Vert T(t_{n})x_{n}-x\right\Vert _{X}=0
$$
for some $t_{n}\to\infty$ and $\{x_{n}\}_{n\in\mathbb{N}}\subset\mathcal{K}$, then $T(t_{n})x_{n}\in\mathcal{K}$
for all $t_{n}>t_{\mathcal{K}}$. By the compactness of $\mathcal{K}$
in $Y$, some subsequence $\{T(t_{n_{j}})x_{n_{j}}\}_{j\in\mathbb{N}}$ converges
in $Y$, which implies that $x\in\omega_{Y}(\mathcal{K})$.

Let $\mathcal{A}:=\omega_{X}(\mathcal{K})=\omega_{Y}(\mathcal{K})$.
Using the notation $\overline{\mathcal{C}}^{Y}$ for the closure of
$\mathcal{C}$ in $Y$, (\ref{eq:omegalimit}) says that
$$
\omega_{Y}(\mathcal{K})=\cap_{t>0}\overline{\cup_{s>t}(T(s)\mathcal{K}\cap Y)}^{Y}=\cap_{t>t_{\mathcal{K}}}\overline{\cup_{s>t}T(s)\mathcal{K}}^{Y},
$$
which implies that $\mathcal{A}$ is a non-empty compact set of $Y$
- and also of $X$ - since it is the intersection of decreasing non-empty compact
sets. The invariance of $\mathcal{A}$ follows easily from (\ref{eq:omegalimit}),
see also \cite[Proposition 1.1.1]{CD}.

Finally, the fact that $\mathcal{A}$ attracts bounded sets of $X$ in $Y$
- and in $X$ as well - follows from the arguments of \cite[Theorem
10.5]{Robinson}, which we provide for completeness. Let us suppose that $\lim_{t\to\infty}\sup_{b\in B}\inf_{a\in\mathcal{A}}\left\Vert T(t)b-a\right\Vert _{Y}=0$
does not hold for some bounded set $B\subset X$. Then, we
can find sequences $t_{n}\to\infty$, $\{b_{n}\}_{n\in\mathbb{N}}\subset B$ and
$\varepsilon_{0}>0$ such that $\inf_{a\in\mathcal{A}}\left\Vert T(t_{n})b_{n}-a\right\Vert_Y >\varepsilon_{0}$
for all $n$. But for $t_{n}>t_{B}$ and $b\in B$, then $T(t_{n})b\in\mathcal{K}$
which is a compact set in $Y$. Hence a subsequence $\{T(t_{n_{j}})b_{n_{j}}\}_{j\in\mathbb{N}}$
converges in $Y$. As $T(t_{n_{j}})b_{n_{j}}=T(t_{n_{j}}-t_{n})T(t_{n})b_{n_{j}}$ and $T(t_{n})b_{n_{j}}\in\mathcal{K}$, we conclude that the subsequence $\{T(t_{n_{j}})b_{n_{j}}\}$ converges to an element of $\omega_{Y}(\mathcal{K})=\mathcal{A}$, which gives us a contradiction. Connectness follows
exactly as in \cite{Robinson}.
\end{proof}
For our purposes we use the following consequence.
\begin{cor}
\label{cor:Tofiindattractors}Let $X$ be a Hilbert space, $A:\mathcal{D}(A)\subset X\to X$
be a negative generator of an analytic semigroup with compact resolvent
and $F:X_{\alpha}=\mathcal{D}((-A)^{\alpha})\to X$, $0\le\alpha<1$,
be a locally Lipschitz function. Consider the following problem 
$$
\begin{aligned}u'(t) & =Au(t)+F(u(t)),\\
u(0) & =u_{0},
\end{aligned}
$$
where $u_{0}\in X_{\alpha}$. Assume that\\
{\em(i)} There is a closed subspace $\widetilde{X}_{\alpha}\subset X_{\alpha}$ such that for all $u_{0}\in\widetilde{X}_{\alpha}$, a global solution $u\in C^{1}((0,\infty),X)\cap C((0,\infty),\mathcal{D}(A))\cap C([0,\infty),\widetilde{X}_{\alpha})$
is defined. \\
{\em(ii)} There are Hilbert spaces $Y$ and $W$ such that $Y\overset{c}{\hookrightarrow}W\hookrightarrow X_{\alpha}$
and a constant $C_{Y}>0$ with the following property: for every $R>0$,
there exists $t_{R}>0$ such that if $u_{0}\in \widetilde{X}_{\alpha}$ and $\left\Vert u_{0}\right\Vert_{X_{\alpha}} \le R$,
then $u(t,u_{0})\in Y$ and $\left\Vert u(t,u_{0})\right\Vert _{Y}\le C_{Y}$,
for all $t>t_{R}$.

Then the semiflow $T:[0,\infty)\times \widetilde{X}_{\alpha}\to \widetilde{X}_{\alpha}$ defined
by $T(t)u_{0}=u(t)$ has a global connected attractor $\mathcal{A}$
that is contained in $W$ and attracts bounded sets of $\widetilde{X}_{\alpha}$
in $W$.
\end{cor}

\begin{proof}
We have to show that the conditions of Theorem \ref{thm:Globalattractor}
are satisfied. As $Y\overset{c}{\hookrightarrow}W$, the bounded set $\mathcal{K}_{0}:=\left\{ x\in Y:\,\left\Vert x\right\Vert _{Y}\le C_{Y}\right\} $
is a relative compact subset of $W$. We define $\mathcal{K}$ to
be the closure of $\mathcal{K}_{0}$ in $W$. Let $B\subset \widetilde{X}_{\alpha}$
be a bounded set, i.e. there exists $R>0$ such that $\left\Vert u_{0}\right\Vert _{X_{\alpha}}\le R$,
for all $u_{0}\in B$. By our assumptions, there exists $t_{R}>0$
such that if $t\ge t_{R}$, then $\left\Vert T(t)u_{0}\right\Vert _{Y}\le C_{Y}$.
Therefore $T(t)B\subset\mathcal{K}_{0}\subset\mathcal{K}$, if $t\ge t_{R}$. 
\end{proof}

The above corollary will now be applied to the proof of the following theorem,
from which part (i) of Theorem \ref{thm:MainTheorem} will follow.

\begin{thm}
\label{thm:Attractorestimate}Let $0<\varepsilon<(n+1)/16-\gamma/8$.
Then, for each $s\ge0$ there is a constant $\varkappa_{s,\varepsilon}>0$
with the following property: for every $R>0$, there exists $\overline{t}_{R,s,\varepsilon}>0$
such that if $u_{0}\in X_{1,0}^{0}$ and $\left\Vert u_{0}\right\Vert _{H_{0}^{-1}(\mathbb{B})}\le R$,
then $\left\Vert u(t,u_{0})\right\Vert _{\mathcal{D}((-A_{s})^{1+\varepsilon})}\le\varkappa_{s,\varepsilon}$,
for all $t>\overline{t}_{R,s,\varepsilon}$.
\end{thm}

The theorem will be proved in several steps. The first one follows
directly from Temam \cite{Temam}. We just highlight the necessary results for repeating the arguments.
\begin{prop}
\label{prop:alpha1/4}There is a constant $\kappa$ with
the following property: for every $R>0$, there is a constant $t_{R}>0$
such that if $u_{0}\in X_{1,0}^{0}$ and $\left\Vert u_{0}\right\Vert _{H_{0}^{-1}(\mathbb{B})}\le R$,
then $\left\Vert u(t,u_{0})\right\Vert _{H_{0}^{1}(\mathbb{B})}\le \kappa$,
for all $t>t_{R}$.
\end{prop}

\begin{proof}
The proof is obtained by following the arguments in \cite[Section 4.2.2]{Temam}. First we prove the existence of $\kappa>0$ with the following
property: for every $R>0$, there is a constant $t_{R}>0$ such that
if $u_{0}\in X_{1,0}^{0}$ and $\left\Vert u_{0}\right\Vert _{H_{0}^{-1}(\mathbb{B})}\le R$,
then $\left\Vert u(t,u_{0})\right\Vert _{H_{0}^{-1}(\mathbb{B})}\le \kappa$,
for all $t>t_{R}$, see deductions of \cite[Equations (4.89)-(4.90)]{Temam}.
For our situation, we only have to take into account the Sobolev immersion
$H^{1}(\mathbb{B})\hookrightarrow L^{4}(\mathbb{B})$ from Corollary
\ref{cor:Mellinsobolevembedding} and Theorem \ref{thm:Gauss-Theorem},
which allow the definition of the strict Lyapunov function \cite[Definition 8.4.5]{HJ} $\mathcal{L}:H^{1}(\mathbb{B})\to\mathbb{R}$ 
by
\begin{equation}\label{eq:Lyapunov}
\mathcal{L}(v)=\frac{1}{2}\int_{\mathbb{B}}\left\langle \nabla v,\nabla v\right\rangle d\mu_{g}+\int_{\mathbb{B}}\left(\frac{1}{4}v^{4}-\frac{1}{2}v^{2}\right)d\mu_{g},
\end{equation}
see also \cite[Section 4.2]{LopesRoidos}. We also use Proposition \ref{prop:HzerobetadualH1}
to identify elements of $\mathcal{H}^{0,\beta}(\mathbb{B})$, $\beta>-1$, with
elements in $H^{-1}(\mathbb{B})$ for the computations. The rest of
proof follows the deduction of \cite[Equation (4.95)]{Temam}.
\end{proof}

In order to proceed to the proof of Theorem \ref{thm:Attractorestimate}, we write \eqref{eq:mainequation}
as
\begin{equation}
u'(t)=A_{s}u+F(u),\label{eq:AbstractCH}
\end{equation}
where $A_{s}:X_{2}^{s}\to X_{0}^{s}$ is given by $A_{s}=-(1-\Delta_{s})^{2}$
and $F:X_{1}^{s}\to X_0^{s}$ is given by $F(u)=\Delta_{s}(u^{3}-3u)+u$. It is well know, see \cite[Theorem 6.13]{Pazy}, that for some $\delta>0$ depending on $A_s$, the fractional powers
satisfy 
\begin{equation}
\left\Vert (-A_{s})^{\alpha}e^{tA_{s}}\right\Vert _{\mathcal{B}\left(\mathcal{H}^{s,\gamma}\left(\mathbb{B}\right)\right)}\le c_{\alpha,s}t^{-\alpha}e^{-\delta t}, \quad t>0,\label{eq:calphas-1}
\end{equation}
where $c_{\alpha,s}>0$ only depends on $\alpha,s\ge 0$.

\begin{lem}\label{lem:induction}
Let $0\le\sigma\le\alpha\le\beta<1$, $0<\tilde{t}<t$, $\delta>0$
be as in \eqref{eq:calphas-1} and $u\in C([\tilde{t},t],\mathcal{D}(A_s))$. Then
\begin{eqnarray}\nonumber
\left\Vert (-A_{s})^{\alpha}u(t)\right\Vert _{\mathcal{H}^{s,\gamma}(\mathbb{B})}
 & \le& C_{\sigma}e^{-\delta(t-\tilde{t})}(t-\tilde{t})^{-\sigma}\left\Vert (-A_{s})^{\alpha-\sigma}u(\tilde{t})\right\Vert _{\mathcal{H}^{s,\gamma}(\mathbb{B})}+C_{\alpha,\beta}\int_{\tilde{t}}^{t}e^{-\delta(t-s)}(t-s)^{-\beta}\\\label{eq:induction}
&&\hspace{-75pt}\times\left(\left\Vert (-A_{s})^{\alpha-\beta+\frac{1}{2}}u^{3}(s)\right\Vert _{\mathcal{H}^{s,\gamma}(\mathbb{B})}+\left\Vert (-A_{s})^{\alpha-\beta+\frac{1}{2}}u(s)\right\Vert _{\mathcal{H}^{s,\gamma}(\mathbb{B})}+\left\Vert (-A_{s})^{\alpha-\beta}u(s)\right\Vert _{\mathcal{H}^{s,\gamma}(\mathbb{B})}\right)ds,
\end{eqnarray}
for some constants $C_{\sigma}$, $C_{\alpha,\beta}$ only depending on $\alpha$, $\beta$, $\sigma$ and $s$.
\end{lem}

\begin{proof}
We apply $(-A_{s})^{\alpha}$ to the variation of constants formula
$$
u(t)=e^{(t-\tilde{t})A_{s}}u(\tilde{t})+\int_{\tilde{t}}^{t}e^{(t-s)A_{s}}F(u(s))ds
$$
to obtain
\begin{eqnarray*} 
\lefteqn{\left\Vert (-A_{s})^{\alpha}u(t)\right\Vert _{\mathcal{H}^{s,\gamma}(\mathbb{B})}}\\
 & \le&\left\Vert (-A_{s})^{\sigma}e^{-A_{s}(t-\tilde{t})}\right\Vert _{\mathcal{B}(\mathcal{H}^{s,\gamma}(\mathbb{B}))}\left\Vert (-A_{s})^{\alpha-\sigma}u(\tilde{t})\right\Vert _{\mathcal{H}^{s,\gamma}(\mathbb{B})}+\int_{\tilde{t}}^{t}\left\Vert (-A_{s})^{\beta}e^{A_{s}(t-s)}\right\Vert _{\mathcal{B}(\mathcal{H}^{s,\gamma}(\mathbb{B}))}\\
 && \quad\times\left\Vert (-A_{s})^{\alpha-\beta}\left(\Delta(u^{3}(s)-3u(s))+u(s)\right)\right\Vert _{\mathcal{H}^{s,\gamma}(\mathbb{B})}ds\\
 & \overset{(1)}{\le}&C_{\sigma}e^{-\delta(t-\tilde{t})}(t-\tilde{t})^{-\sigma}\left\Vert (-A_{s})^{\alpha-\sigma}u(\tilde{t})\right\Vert _{\mathcal{H}^{s,\gamma}(\mathbb{B})}+C_{\alpha,\beta}\int_{\tilde{t}}^{t}e^{-\delta(t-s)}(t-s)^{-\beta}\\
 && \quad\times\left(\left\Vert (-A_{s})^{\alpha-\beta+\frac{1}{2}}u^{3}(s)\right\Vert _{\mathcal{H}^{s,\gamma}(\mathbb{B})}+\left\Vert (-A_{s})^{\alpha-\beta+\frac{1}{2}}u(s)\right\Vert _{\mathcal{H}^{s,\gamma}(\mathbb{B})}+\left\Vert (-A_{s})^{\alpha-\beta}u(s)\right\Vert _{\mathcal{H}^{s,\gamma}(\mathbb{B})}\right)ds.
\end{eqnarray*}
In $(1)$ we have used that 
$$
\begin{aligned} & \left\Vert (-A_{s})^{\alpha-\beta}\Delta v\right\Vert _{\mathcal{H}^{s,\gamma}(\mathbb{B})} =\left\Vert (1-\Delta_{s})^{-1}\Delta_{s}(-A_{s})^{\alpha-\beta+\frac{1}{2}}v\right\Vert _{\mathcal{H}^{s,\gamma}(\mathbb{B})}\le c_{s}\left\Vert (-A_{s})^{\alpha-\beta+\frac{1}{2}}v\right\Vert _{\mathcal{H}^{s,\gamma}(\mathbb{B})}.
\end{aligned}
$$
\end{proof}
\begin{prop}
\label{prop:Estimatefor1/2}There is a constant $\kappa_1>0$
with the following property: for every $R>0$, there is a constant
$t_{R,1}>0$ such that if $u_{0}\in X_{1,0}^{0}$ and $\left\Vert u_{0}\right\Vert _{H_{0}^{-1}(\mathbb{B})}\le R$,
then $\left\Vert u(t,u_{0})\right\Vert _{X_{1}^{0}}\le \kappa_1$,
for all $t>t_{R,1}$.

\end{prop}

\begin{proof}
\emph{Step 1}. Let $\theta\in[1/2,1)$. There exists $\kappa_\theta>0$
with the following property: for every $R>0$, there is a constant
$t_{R,\theta}>0$ such that if $u_{0}\in X_{1,0}^{0}$ and $\left\Vert u_{0}\right\Vert _{H_{0}^{-1}(\mathbb{B})}\le R$,
then $\left\Vert u(t,u_{0})\right\Vert _{X_{\theta}^{0}}\le \kappa_\theta$,
for all $t>t_{R,\theta}$.

Let $u_{0}\in X_{1,0}^{0}$ be such that $\left\Vert u_{0}\right\Vert _{H_{0}^{-1}(\mathbb{B})}\le R$,
$u(t)=T(t)u_{0}$ and $t_{R}$, $\kappa$ be as Proposition
\ref{prop:alpha1/4}. If $\mathsf{x}$ is as in Remark \ref{rem:xmellinL2weight}, then we have
\begin{equation}
\left\Vert u(s)^{3}\right\Vert _{\mathcal{H}^{0,\gamma}(\mathbb{B})}\le C\left(\int_{\mathbb{B}}\left|\mathsf{x}^{-\gamma/3}u(s)\right|^{6}d\mu_{g}\right)^{1/2}=C\left\Vert u\right\Vert _{\mathsf{x}^{\gamma/3}L^{6}(\mathbb{B})}^{3}\le C\left\Vert u\right\Vert _{H^{1}(\mathbb{B})}^{3},\label{eq:u3H0,gamma=00005CleH1}
\end{equation}
where we have used Corollary \ref{cor:Mellinsobolevembedding} and \eqref{gamma}. Also, due to Proposition \ref{prop:inclusionH1mellin} and Proposition \ref{thm:Interpolation} (i),
for suitable $0\le\ell<1$, we have $\gamma +\ell<1$ and
$$
H^{1}(\mathbb{B})\hookrightarrow\mathcal{H}^{\ell,\gamma+\ell}(\mathbb{B})=\mathcal{H}^{\ell,\gamma+\ell}(\mathbb{B})\oplus\mathbb{C}_{\omega}=X_{\ell/2}^{0}.
$$
Hence, using \eqref{eq:induction} with $\tilde{t}=t_{R}$, $\alpha=\frac{\theta}{2}$,
$\sigma=\frac{\theta}{2}-\frac{1}{4}+\varepsilon$ for some $\varepsilon>0$,
 $\beta=\frac{\theta}{2}+\frac{1}{2}$, \eqref{eq:u3H0,gamma=00005CleH1} and Proposition \ref{prop:alpha1/4},
we obtain for $t>t_R$
\begin{eqnarray*} \lefteqn{\left\Vert (-A_{0})^{\frac{\theta}{2}}u(t)\right\Vert _{\mathcal{H}^{0,\gamma}(\mathbb{B})} \le Ce^{-\delta(t-\tilde{t})}(t-\tilde{t})^{-(\frac{\theta}{2}-\frac{1}{4}+\varepsilon)}\left\Vert (-A_{0})^{\frac{1}{4}-\varepsilon}u(\tilde{t})\right\Vert _{\mathcal{H}^{0,\gamma}(\mathbb{B})}}\\
 && + C\int_{\tilde{t}}^{t}e^{-\delta(t-s)}(t-s)^{-(\frac{\theta}{2}+\frac{1}{2})}\left(\left\Vert u^{3}(s)\right\Vert _{\mathcal{H}^{0,\gamma}(\mathbb{B})}+\left\Vert u(s)\right\Vert _{\mathcal{H}^{0,\gamma}(\mathbb{B})}+\left\Vert (-A_{0})^{-\frac{1}{2}}u(s)\right\Vert _{\mathcal{H}^{0,\gamma}(\mathbb{B})}\right)ds\\
 & \le& Ce^{-\delta(t-\tilde{t})}(t-\tilde{t})^{-(\frac{\theta}{2}-\frac{1}{4}+\varepsilon)}\left\Vert u(\tilde{t})\right\Vert _{H^{1}(\mathbb{B})}\\
 && +C\int_{\tilde{t}}^{t}e^{-\delta(t-s)}(t-s)^{-(\frac{\theta}{2}+\frac{1}{2})}\left(\left\Vert u(s)\right\Vert _{H^{1}(\mathbb{B})}^{3}+\left\Vert u(s)\right\Vert _{H^{1}(\mathbb{B})}\right)ds\\
 & \le& Ce^{-\delta(t-t_{R})}(t-t_{R})^{-(\frac{\theta}{2}-\frac{1}{4}+\varepsilon)}+C\int_{0}^{\infty}e^{-\delta s}s^{-\frac{\theta}{2}-\frac{1}{2}}ds,
\end{eqnarray*}
where the contants $C$ in the last line depend on $\kappa$, since $\Vert u(t)\Vert_{H^{1}(\mathbb{B})}\le\kappa$ for $t\ge t_{R}$.

Let us define $\kappa_\theta:=C+C\int_{0}^{\infty}e^{-\delta s}s^{-\theta/2-1/2}ds$
and choose $t_{R,\theta}>t_{R}$ such that
$$
e^{-\delta(t_{R,\theta}-t_{R})}(t_{R,\theta}-t_{R})^{-(\frac{\theta}{2}-\frac{1}{4}+\varepsilon)}<1.
$$
Then $\Vert (-A_{0})^{\theta/2}u(t)\Vert _{\mathcal{H}^{0,\gamma}(\mathbb{B})}\le \kappa_\theta$,
$\forall t>t_{R,\theta}$.\\
\emph{Step 2}. Choose in \eqref{eq:induction} $\alpha=\frac{1}{2}$,
$\sigma=\frac{1}{4}$ and $\beta$, such that $\frac{1}{2}<\beta<1+\frac{\gamma}{4}-\frac{n+1}{8}$.
This is possible as $n\in\{1,2\}$ and $\frac{n-3}{2}<\gamma\le0$.
Hence $\frac{n+1}{2}-\gamma<2$, which implies that $\frac{\gamma}{4}-\frac{n+1}{8}>-\frac{1}{2}$.
With this choice of $\beta$, we also have $\frac{1}{2}>1-\beta>\frac{n+1}{8}-\frac{\gamma}{4}$.
Therefore, with a suitable choice of $\beta$, we have, according to Proposition \ref{thm:Interpolation} (i), that $X_{2(1-\beta)}^{0}=\mathcal{H}^{4(1-\beta),4(1-\beta)+\gamma}(\mathbb{B})\oplus\mathbb{C}_{\omega}$ is an algebra as $4(1-\beta)+\gamma>\frac{n+1}{2}$ and $4(1-\beta)>\frac{n+1}{2}$,
as $\gamma\le0$.

Choosing $\tilde{t}=\max\{t_{R,2(1-\beta)},t_{R,1/2}\}$ and using 
\eqref{eq:induction} with $t>\tilde{t}$, $\alpha=1/2$, $\sigma=1/4$ and $\beta$ as above, we have
\begin{eqnarray*}\lefteqn{\left\Vert (-A_{0})^{\frac{1}{2}}u(t)\right\Vert _{\mathcal{H}^{0,\gamma}(\mathbb{B})} \le Ce^{-\delta(t-\tilde{t})}(t-\tilde{t})^{-\frac{1}{4}}\left\Vert (-A_{0})^{\frac{1}{4}}u(\tilde{t})\right\Vert _{\mathcal{H}^{0,\gamma}(\mathbb{B})}+C\int_{\tilde{t}}^{t}e^{-\delta(t-s)}(t-s)^{-\beta}}\\
 && \times\left(\left\Vert (-A_{0})^{1-\beta}u^{3}(s)\right\Vert _{\mathcal{H}^{0,\gamma}(\mathbb{B})}+\left\Vert (-A_{0})^{1-\beta}u(s)\right\Vert _{\mathcal{H}^{0,\gamma}(\mathbb{B})}+\left\Vert (-A_{0})^{\frac{1}{2}-\beta}u(s)\right\Vert _{\mathcal{H}^{0,\gamma}(\mathbb{B})}\right)ds\\
 & \le& Ce^{-\delta(t-\tilde{t})}(t-\tilde{t})^{-\frac{1}{4}}\left\Vert u(\tilde{t})\right\Vert _{X_{1/2}^{0}}+C\int_{\tilde{t}}^{t}e^{-\delta(t-s)}(t-s)^{-\beta}\\
& & \times(\left\Vert u(s)\right\Vert _{X_{2(1-\beta)}^{0}}^{3}+\left\Vert u(s)\right\Vert _{X_{2(1-\beta)}^{0}})ds\\
 & \le& Ce^{-\delta(t-\tilde{t})}(t-\tilde{t})^{-\frac{1}{4}}\kappa_{1/2}+C\int_{0}^{\infty}e^{-\delta s}s^{-\beta}ds\left(\kappa_{2(1-\beta)}^{3}+\kappa_{2(1-\beta)}\right).
\end{eqnarray*}
Let us define
$$
\kappa_{1}:=C\kappa_{1/2}+C\left(\int_{0}^{\infty}e^{-\delta s}s^{-\beta}ds\right)\left(\kappa_{2(1-\beta)}^{3}+\kappa_{2(1-\beta)}\right)
$$
and we choose $t_{R,1}$ such that $t_{R,1}>\tilde{t}$ and $C\kappa_{1/2}e^{-\delta(t-\tilde{t})}(t-\tilde{t})^{-\frac{1}{4}}<1$,
for $t>t_{R,1}$. Therefore, we conclude that for $t>t_{R,1}$, we
have $\Vert (-A_{0})^{\frac{1}{2}}u(t)\Vert _{\mathcal{H}^{0,\gamma}(\mathbb{B})}\le \kappa_{1}$.
\end{proof}

\begin{prop}
\label{prop:Attractorestimatetheta}For every $s\ge0$
and $\theta\in[1/2,1)$ there is a constant $\kappa_{s,\theta}>0$
with the following property: for every $R>0$ there is a constant
$t_{R,s,\theta}>0$ such that, if $u_{0}\in X_{1,0}^{0}$ and $\left\Vert u_{0}\right\Vert _{H_{0}^{-1}(\mathbb{B})}\le R$,
then $\left\Vert u(t,u_{0})\right\Vert _{\mathcal{D}((-A_{s})^{\theta})}\le\kappa_{s,\theta}$,
for all $t>t_{R,s,\theta}$.
\end{prop}

\begin{proof}
The result is a direct consequence of the following two claims.\\
\emph{First claim}: Suppose that for some $s\ge0$ there is a constant
$\kappa_{s}>0$ with the following property: for every $R>0$ there
exists $\tilde{t}_{R,s}>0$ such that, if $u_{0}\in X_{1,0}^{0}$ and
$\left\Vert u_{0}\right\Vert _{H_{0}^{-1}(\mathbb{B})}\le R$, then
$\left\Vert u(t,u_{0})\right\Vert _{\mathcal{H}^{s+2,\gamma+2}(\mathbb{B})\oplus\mathbb{C}_{\omega}}\le\kappa_{s}$,
for all $t>\tilde{t}_{R,s}$. If such a constant $\kappa_{s}>0$ exists, then for each $\theta\in[1/2,1)$ there
is also a constant $\kappa_{s,\theta}>0$ with the following property:
for every $R>0$ there exists $t_{R,s,\theta}>0$ such that, if $u_{0}\in X_{1,0}^{0}$
and $\left\Vert u_{0}\right\Vert _{H_{0}^{-1}(\mathbb{B})}\le R$,
then $\left\Vert u(t,u_{0})\right\Vert _{\mathcal{D}((-A_{s})^{\theta})}\le\kappa_{s,\theta}$,
for all $t>t_{R,s,\theta}$.\\
\emph{Proof of the first claim}: We use \eqref{eq:induction}
with $\alpha=\beta=\sigma=\theta\in[1/2,1)$, $\tilde{t}=\tilde{t}_{R,s}$
and $t>\tilde{t}$, to obtain
\begin{eqnarray*}
\left\Vert (-A_{s})^{\theta}u(t)\right\Vert _{\mathcal{H}^{s,\gamma}(\mathbb{B})} & \le & C_{\sigma}e^{-\delta(t-\tilde{t})}(t-\tilde{t})^{-\theta}\left\Vert u(\tilde{t})\right\Vert _{\mathcal{H}^{s,\gamma}(\mathbb{B})}+C_{\alpha,\beta}\int_{\tilde{t}}^{t}e^{-\delta(t-s)}(t-s)^{-\theta}\\
 & & \times\left(\left\Vert (-A_{s})^{\frac{1}{2}}u^{3}(s)\right\Vert _{\mathcal{H}^{s,\gamma}(\mathbb{B})}+\left\Vert (-A_{s})^{\frac{1}{2}}u(s)\right\Vert _{\mathcal{H}^{s,\gamma}(\mathbb{B})}+\left\Vert u(s)\right\Vert _{\mathcal{H}^{s,\gamma}(\mathbb{B})}\right)ds\\
 & \le & C_{\sigma}e^{-\delta(t-\tilde{t})}(t-\tilde{t})^{-\theta}\left\Vert u(\tilde{t})\right\Vert _{\mathcal{H}^{s,\gamma}(\mathbb{B})}+C_{\alpha,\beta}\int_{\tilde{t}}^{t}e^{-\delta(t-s)}(t-s)^{-\theta}\\
 & & \times\left(\left\Vert u(s)\right\Vert _{\mathcal{H}^{s+2,\gamma+2}(\mathbb{B})\oplus\mathbb{C}_{\omega}}^{3}+\left\Vert u(s)\right\Vert _{\mathcal{H}^{s+2,\gamma+2}(\mathbb{B})\oplus\mathbb{C}_{\omega}}+\left\Vert u(s)\right\Vert _{\mathcal{H}^{s,\gamma}(\mathbb{B})}\right)ds\\
 & \le & C_{\sigma}(t-\tilde{t}_{R,s})^{-\theta}e^{-\delta(t-\tilde{t}_{R,s})}\kappa_{s}+C_{\alpha,\beta}\left(\kappa_{s}^{3}+2\kappa_{s}\right)\int_{0}^{\infty}s^{-\theta}e^{-\delta s}ds.
\end{eqnarray*}
Let us choose
$$
\kappa_{s,\theta}:=C_{\sigma}\kappa_{s}+C_{\alpha,\beta}\left(\kappa_{s}^{3}+2\kappa_{s}\right)\int_{0}^{\infty}s^{-\theta}e^{-\delta s}ds
$$
and $t_{R,s,\theta}>\tilde{t}_{R,s}$ such that $(t-\tilde{t}_{R,s})^{-\theta}e^{-\delta(t-\tilde{t}_{R,s})}\le1$
when $t\ge t_{R,s,\theta}$. Then $\left\Vert u(t)\right\Vert _{\mathcal{D}((-A_{s})^{\theta})}\le\kappa_{s,\theta}$
for all $t>t_{R,s,\theta}$.\\
\emph{Second claim}: For every $s\ge0$ there is a constant $\kappa_{s}>0$
with the following property: for every $R>0$, there exists $\tilde{t}_{R,s}>0$
such that, if $u_{0}\in X_{1,0}^{0}$ and $\left\Vert u_{0}\right\Vert _{H_{0}^{-1}(\mathbb{B})}\le R$,
then $\left\Vert u(t,u_{0})\right\Vert _{\mathcal{H}^{s+2,\gamma+2}(\mathbb{B})\oplus\mathbb{C}_{\omega}}\le\kappa_{s}$,
for all $t>\tilde{t}_{R,s}$.\\
\emph{Proof of the second claim}: We have seen that this is true for
$s=0$ by Proposition \ref{prop:Estimatefor1/2}.
We now proceed by induction as follows: we prove that, if the property
holds for some $s_{0}\ge0$, then it also holds for all $s\in[s_{0},s_{0}+1]$. 
Indeed, let us suppose that it holds for some $s_{0}\ge0$. Taking $\theta>3/4$, $\sigma\in[0,1]$,
$s=s_{0}+\sigma$ and a suitable small $\varepsilon>0$,
Proposition \ref{thm:Interpolation} (iii) with $\alpha = 3/4 +\varepsilon$ implies
$$
\begin{aligned} & \left\Vert u(t,u_{0})\right\Vert _{\mathcal{H}^{s_{0}+\sigma,\gamma+2}(\mathbb{B})\oplus\mathbb{C}_{\omega}}\le\left\Vert u(t,u_{0})\right\Vert _{\mathcal{H}^{s_{0}+3,\gamma+2}(\mathbb{B})\oplus\mathbb{C}_{\omega}}\le\left\Vert u(t,u_{0})\right\Vert _{\mathcal{H}^{s_{0}+3,\gamma+3}(\mathbb{B})\oplus\mathbb{C}_{\omega}\oplus\underline{\mathcal{E}}_{\Delta^{2},\gamma-1+4\varepsilon}}\\
 & \le\left\Vert u(t,u_{0})\right\Vert _{[\mathcal{H}^{s_{0},\gamma}(\mathbb{B}),X_{2}^{s_{0}}]_{3/4+\varepsilon}}=\left\Vert u(t,u_{0})\right\Vert _{\mathcal{D}((-A_{s_{0}})^{3/4+\varepsilon})}\le\left\Vert u(t,u_{0})\right\Vert _{\mathcal{D}((-A_{s_{0}})^{\theta})}.
\end{aligned}
$$
By the induction hypothesis and the first claim, the last term is
smaller or equal to $\kappa_{s_{0},\theta}$ for all $t>t_{R,s_{0},\theta}$.
Hence the result follows for $\tilde{t}_{R,s_{0}+\sigma}:=t_{R,s_{0},\theta}$,
for all $\sigma\in(0,1]$.
\end{proof}

\begin{proof}
(of Theorem \ref{thm:Attractorestimate}) First we
note that, choosing $\varepsilon>0$ properly, we have
\begin{equation}
\begin{aligned}\mathcal{D}((-A_{s})^{2\varepsilon}) & =[\mathcal{H}^{s,\gamma}(\mathbb{B}),\mathcal{D}(-A_{s})]_{2\varepsilon}=[\mathcal{H}^{s,\gamma}(\mathbb{B}),\mathcal{D}((-A_{s})^{1/2})]_{4\varepsilon}\\
 & =[\mathcal{H}^{s,\gamma}(\mathbb{B}),\mathcal{H}^{s+2,\gamma+2}(\mathbb{B})\oplus\mathbb{C}_{\omega}]_{4\varepsilon}\overset{(1)}{=}\mathcal{H}^{s+8\varepsilon,\gamma+8\varepsilon}(\mathbb{B})\oplus\mathbb{C}_{\omega}=\mathcal{H}^{s+8\varepsilon,\gamma+8\varepsilon}(\mathbb{B}),
\end{aligned}
\label{eq:varepsilon/2}
\end{equation}
where we have used Proposition \ref{thm:Interpolation} (i) in (1) and that $\gamma+8\varepsilon<(n+1)/2$ in the last equality.
Moreover for suitable $0<\tilde{\varepsilon}<1/2-2\varepsilon$ we have
\begin{align}\label{eq:(1+E)2}\nonumber
\mathcal{D}((-A_{s})^{1/2+2\varepsilon+\tilde{\varepsilon}}) & =[\mathcal{H}^{s,\gamma}(\mathbb{B}),\mathcal{D}(A_{s})]_{1/2+2\varepsilon+\tilde{\varepsilon}}\\
 & \hookrightarrow\mathcal{H}^{s+2+8\varepsilon,\gamma+2+8\varepsilon}(\mathbb{B})\oplus\mathbb{C}_{\omega}\oplus\underline{\mathcal{E}}_{\Delta^{2},\gamma-2+8\varepsilon+4\tilde{\varepsilon}},
\end{align}
where we have used Proposition \ref{thm:Interpolation} (iii).

Let $t_{R,s,\theta}>0$, $\theta=1/2+2\varepsilon+\tilde{\varepsilon}$
be as in Proposition \ref{prop:Attractorestimatetheta},
and $u_{0}\in X_{1,0}^{0}$. Then, applying formally $(-A_{s})^{1+\varepsilon}$
to the variation of constants formula give us 
\begin{equation}
(-A_{s})^{1+\varepsilon}u(t)=(-A_{s})^{1+\varepsilon}e^{(t-t_{R,s,\theta})A_{s}}u(t_{R,s,\theta})+\int_{t_{R,s,\theta}}^{t}(-A_{s})^{1-\varepsilon}e^{(t-s)A_{s}}(-A_{s})^{2\varepsilon}F(u(s))ds.\label{eq:variation1+epsilon}
\end{equation}
Notice however that we do not know that $u(t)\in\mathcal{D}((-A_{s})^{1+\varepsilon})$
a priori. This will follow by showing that the $\mathcal{H}^{s,\gamma}(\mathbb{B})$
norm of the integrand of (\ref{eq:variation1+epsilon})
is integrable, see \cite[Proposition 1.1.7]{Arendtetal},
which is a consequence of the following computations, similar to Lemma
\ref{lem:induction}. 
\begin{eqnarray}
\left\Vert u(t)\right\Vert _{\mathcal{D}((-A_{s})^{1+\varepsilon})} & = & \left\Vert (-A_{s})^{1+\varepsilon}u(t)\right\Vert _{\mathcal{H}^{s,\gamma}(\mathbb{B})}\nonumber \\
 & \le & C\left\Vert (-A_{s})^{1+\varepsilon}e^{(t-t_{R,s,\theta})A_{s}}\right\Vert _{\mathcal{B}(\mathcal{H}^{s,\gamma}(\mathbb{B}))}\left\Vert u(t_{R,s,\theta})\right\Vert _{\mathcal{H}^{s,\gamma}(\mathbb{B})}\nonumber \\
 & & +C\int_{t_{R,s,\theta}}^{t}\left\Vert (-A_{s})^{1-\varepsilon}e^{(t-s)A_{s}}\right\Vert _{\mathcal{B}(\mathcal{H}^{s,\gamma}(\mathbb{B}))}\left\Vert (-A_{s})^{2\varepsilon}F(u(s))\right\Vert _{\mathcal{H}^{s,\gamma}(\mathbb{B})}ds.\label{eq:variaationsinduction-1}
\end{eqnarray}
Notice that
\begin{eqnarray}
\left\Vert (-A_{s})^{2\varepsilon}F(u(s))\right\Vert _{\mathcal{H}^{s,\gamma}(\mathbb{B})} & \overset{(1)}{=} & \left\Vert \Delta_{s}(u^{3}(s)-3u(s))+u(s)\right\Vert _{\mathcal{H}^{s+8\varepsilon,\gamma+8\varepsilon}(\mathbb{B})}\nonumber \\
 & \overset{(2)}{\le} & C\left\Vert u^{3}(s)-3u(s)\right\Vert _{\mathcal{H}^{s+8\varepsilon+2,\gamma+8\varepsilon+2}(\mathbb{B})\oplus\mathbb{C}_{\omega}}+C\left\Vert u(s)\right\Vert _{\mathcal{H}^{s+8\varepsilon,\gamma+8\varepsilon}(\mathbb{B})}\nonumber \\
 & \le & C(\left\Vert u(s)\right\Vert _{\mathcal{H}^{s+8\varepsilon+2,\gamma+8\varepsilon+2}(\mathbb{B})\oplus\mathbb{C}_{\omega}\oplus \underline{\mathcal{E}}_{\Delta^{2},\gamma-2+8\varepsilon+4\tilde{\varepsilon}}}^{3}+\left\Vert u(s)\right\Vert _{\mathcal{H}^{s+8\varepsilon,\gamma+8\varepsilon}(\mathbb{B})})\nonumber \\
 & \overset{(3)}{\le} & C(\left\Vert u(s)\right\Vert _{\mathcal{D}((-A_{s})^{1/2+2\varepsilon+\tilde{\varepsilon}})}^{3}+\left\Vert u(s)\right\Vert _{\mathcal{D}((-A_{s})^{1/2+2\varepsilon+\tilde{\varepsilon}})})\nonumber \\
 & \overset{(4)}{\le} & C(\kappa_{s,1/2+2\varepsilon+\tilde{\varepsilon}}^{3}+\kappa_{s,1/2+2\varepsilon+\tilde{\varepsilon}}),\label{eq:variaationsinductionF-1}
\end{eqnarray}
where we have used \eqref{eq:varepsilon/2} in (1),
the continuity of $\Delta:\mathcal{H}^{s+8\varepsilon+2,\gamma+8\varepsilon+2}(\mathbb{B})\oplus\mathbb{C}_{\omega}\to\mathcal{H}^{s+8\varepsilon,\gamma+8\varepsilon}(\mathbb{B})$
in (2), \eqref{eq:(1+E)2} in (3) and Proposition
\ref{prop:Attractorestimatetheta} in (4). By \eqref{eq:variaationsinduction-1}
and \eqref{eq:variaationsinductionF-1}, we find
$$
\left\Vert u(t)\right\Vert _{\mathcal{D}((-A_{s})^{1+\varepsilon})}\le C(t-t_{R,s,\theta})^{-(1+\varepsilon)}e^{-\delta(t-t_{R,s,\theta})}\kappa_{s,\theta}+C\left(\kappa_{s,\theta}^{3}+\kappa_{s,\theta}\right)\int_{0}^{\infty}s^{-(1-\varepsilon)}e^{-\delta s}ds.
$$
 Let us choose 
$$
\varkappa_{s,\varepsilon}:=C\kappa_{s,\theta}+C\left(\kappa_{s,\theta}^{3}+\kappa_{s,\theta}\right)\int_{0}^{\infty}s^{-(1-\varepsilon)}e^{-\delta s}ds
$$
 and $\overline{t}_{R,s,\varepsilon}>t_{R,s,\theta}$ such that $(t-t_{R,s,\theta})^{-(1+\varepsilon)}e^{-\delta(t-t_{R,s,\theta})}\le1$
for $t\ge\overline{t}_{R,s,\varepsilon}$. Then $\left\Vert u(t)\right\Vert _{\mathcal{D}((-A_{s})^{1+\varepsilon})}\le\varkappa_{s,\varepsilon}$
for all $t>\overline{t}_{R,s,\varepsilon}$. 
\end{proof}

We are now finally in position to prove part (i) of Theorem \ref{thm:MainTheorem}.
\begin{proof}
(of part (i) of Theorem \ref{thm:MainTheorem}). We check the conditions
of Corollary \ref{cor:Tofiindattractors} for \eqref{eq:AbstractCH}. Here we
use $\alpha=1/2$, so that $X_{\alpha}=X_{1}^{s}$, and $\widetilde{X}_{\alpha}=X_{1,0}^{s}$. For any $r\ge s$,
we choose $W=\mathcal{D}(\Delta_{r}^{2})$ and $Y=\mathcal{D}((-A_{r})^{1+\varepsilon})$, where $\varepsilon$ is as in Theorem \ref{thm:Attractorestimate}.
Condition (i) follows from \eqref{extrareg}-\eqref{contsolu} and Proposition \ref{prop:avaragedoesnotchange}.
For condition (ii), we first note that 
$$
\mathcal{D}((-A_{r})^{1+\varepsilon})\overset{c}{\hookrightarrow}\mathcal{D}(-A_{r})=\mathcal{D}(\Delta_{r}^{2})\hookrightarrow\mathcal{H}^{s+2,\gamma+2}(\mathbb{B})\oplus\mathbb{C}_{\omega}.
$$
Moreover, if $u_{0}\in X_{1,0}^{s}$ and $\left\Vert u_{0}\right\Vert _{X_{1}^{s}}\le R$, then,
as $X_{1}^{s}\hookrightarrow H^{-1}(\mathbb{B})$ and $\int_{\mathbb{B}} u_{0}d\mu_{g}=0$,
we conclude that $\left\Vert u_{0}\right\Vert _{H_{0}^{-1}(\mathbb{B})}\le \widetilde{R}$.
Theorem \ref{thm:Attractorestimate} gives the necessary estimate of the second condition.

Corollary \ref{cor:Tofiindattractors} implies the existence of a connected global attractor $\mathcal{A}^{s}$
for the semiflow $T:[0,\infty)\times X_{1,0}^{s}\to X_{1,0}^{s}$.
By uniqueness of the global attractor, $\mathcal{A}^{s}$ does not
depend on $r$. Hence $\mathcal{A}^{s}\subset\mathcal{D}(\Delta_{r}^{2})$
for all $r>0$ and it attracts bounded sets of $X_{1,0}^{s}$
in $\mathcal{D}(\Delta_{r}^{2})$.

For the $s$-independence, let $s_{1}>s_{2}\ge0$. As $X_{1}^{s_{1}}\hookrightarrow X_{1}^{s_{2}}$
is continuous and $\mathcal{A}^{s_{1}}$ is compact in $X_{1,0}^{s_{1}}$,
we conclude that $\mathcal{A}^{s_{1}}$ is also compact in $X_{1,0}^{s_{2}}$.
Consider now a bounded set $B\subset X_{1,0}^{s_{2}}$. Due to Theorem
\ref{thm:Attractorestimate}, there exists $\tilde{t}>0$ such that
the set $T(\tilde{t})B$ is a bounded set of $X_{1,0}^{s_{1}}$. Therefore,
for $t\ge\tilde{t}$, we have 
\begin{eqnarray*}\sup_{b\in B}\inf_{a\in\mathcal{A}^{s_{1}}}\left\Vert T(t)b-a\right\Vert _{X_{1}^{s_{2}}} & \le&\sup_{b\in B}\inf_{a\in\mathcal{A}^{s_{1}}}\left\Vert T(t)b-a\right\Vert _{X_{1}^{s_{1}}}\\
 & =&\sup_{b\in B}\inf_{a\in\mathcal{A}^{s_{1}}}\left\Vert T(t-\tilde{t})T(\tilde{t})b-a\right\Vert _{X_{1}^{s_{1}}}\overset{t\to\infty}{\longrightarrow}0.
\end{eqnarray*}
 Finally, as $T(t)\mathcal{A}^{s_{1}}=\mathcal{A}^{s_{1}}$, we conclude
that $\mathcal{A}^{s_{1}}$ is a global attractor for the semiflow $T$
in $X_{1,0}^{s_{2}}$. By uniqueness of global attractors $\mathcal{A}^{s_{1}}=\mathcal{A}^{s_{2}}$.
\end{proof}

\section{Convergence to the equilibrium}

In this section, we prove part (ii) of Theorem \ref{thm:MainTheorem}. We first state an abstract result from \cite{HJ}. Let $V$ and $H$ be real Hilbert spaces such that $V$ is densely
and continuously embedded to $H$. We recall that an element $x\in H$
defines a continuous linear functional in $V$ by $y\in V\mapsto(y,x)_{H}\in\mathbb{R}$.
Under this, we have $V\overset{i_{V,H}}{\hookrightarrow}H\overset{i_{H,V^{*}}}{\hookrightarrow}V^{*}$,
where $i_{V,H}$ and $i_{H,V^{*}}$ are continuous embeddings with
dense image.
\begin{thm}
\label{thm:HarauxJendoubi} \cite[Section 11.2]{HJ}
Let $E:V\to\mathbb{R}$ be a real analytic function such that $E(0)=0\in\mathbb{R}$,
$DE(0)=0\in V^{*}$ and $A:=D^{2}E(0):V\to V^{*}$ is a Fredholm operator,
where $DE:V\to V^{*}$ and $D^{2}E:V\to\mathcal{B}(V,V^{*})$ are
the first and second Fr\'echet derivatives. Then there exist $\theta\in(0,1/2]$,
$\sigma>0$ and $c>0$ such that 
$$
\left|E(v)\right|^{1-\theta}\le c\left\Vert DE(v)\right\Vert _{V^{*}}\quad\text{for all }v\in V\text{ satisfying }\left\Vert v\right\Vert _{V}<\sigma.
$$
\end{thm}

The above inequality is called the Lojasiewicz-Simon inequality at
$0$. In our application, the function $E$ of Theorem \ref{thm:HarauxJendoubi}
will be related to the Lyapunov (energy) functional defined for the
Cahn-Hilliard equation. In this section, we always assume that $\dim(\mathbb{B})\in\left\{ 2,3\right\} $
and work with the subspaces of real functions. In order to apply the Theorem \ref{thm:HarauxJendoubi}, we need the following
technical lemma.

\begin{lem}
\label{prop:uemu2v} If $u\in H^{1}(\mathbb{B})$, then the linear
operator $T_{u}:H^{1}(\mathbb{B})\to H^{-1}(\mathbb{B})$ defined
by 
$$
\left\langle T_{u}(v),h\right\rangle _{H^{-1}(\mathbb{B})\times H^{1}(\mathbb{B})}=\int_{\mathbb{B}}u^{2}vhd\mu_{g}
$$
is continuous and compact.
\end{lem}

\begin{proof}
Let $\beta>0$ and $\widetilde{T}_{u}:H^{1}(\mathbb{B})\to\mathcal{H}^{0,-\beta}(\mathbb{B})$
be defined by $\widetilde{T}_{u}(v)=u^{2}v$. This function is continuous. In fact,
$$
\begin{aligned}\int_{\mathbb{B}}\mathsf{x}^{2\beta}\left|u^{2}v\right|^{2}d\mu_{g} & =\int_{\mathbb{B}}\mathsf{x}^{\beta}u^{4}\mathsf{x}^{\beta}v^{2}d\mu_{g}\overset{(1)}{\le}\Vert \mathsf{x}^{\frac{\beta}{4}}u\Vert _{L^{6}(\mathbb{B})}^{4}\Vert \mathsf{x}^{\frac{\beta}{2}}v\Vert _{L^{6}(\mathbb{B})}^{2}\overset{(2)}{\le}C\left\Vert u\right\Vert _{H^{1}(\mathbb{B})}^{4}\left\Vert v\right\Vert _{H^{1}(\mathbb{B})}^{2}.\end{aligned}
$$
In $(1)$ we have used H\"older inequality and in $(2)$
Corollary \ref{cor:Mellinsobolevembedding}. Therefore we have $\Vert \widetilde{T}_{u}(v)\Vert _{\mathcal{H}^{0,-\beta}(\mathbb{B})}\le C\left\Vert u\right\Vert _{H^{1}(\mathbb{B})}^{2}\left\Vert v\right\Vert _{H^{1}(\mathbb{B})}$.

Let us fix $0<\beta<\alpha<1$. We observe that 
$$
H^{1}(\mathbb{B})\hookrightarrow\mathcal{H}^{1,\alpha}(\mathbb{B})\overset{c}{\hookrightarrow}\mathcal{H}^{0,\beta}(\mathbb{B})\hookrightarrow\mathcal{H}^{0,0}(\mathbb{B}).
$$ 
Therefore $\mathcal{H}^{0,\beta}(\mathbb{B})^{*}\overset{c}{\hookrightarrow}H^{-1}(\mathbb{B})$
is also compact. Denote by $i_{\mathcal{A},\mathcal{B}}:\mathcal{A}\to\mathcal{B}$ the inclusion map $\mathcal{A}\hookrightarrow\mathcal{B}$
 and by
$I:\mathcal{H}^{0,-\beta}(\mathbb{B})\to\mathcal{H}^{0,\beta}(\mathbb{B})^{*}$
the usual identification induced by the inner product in $\mathcal{H}^{0,0}(\mathbb{B})$.
The following map
$$
i_{\mathcal{H}^{0,\beta}(\mathbb{B})^{*},H^{1}(\mathbb{B})^{*}}\circ I\circ\widetilde{T}_{u}:H^{1}(\mathbb{B})\to\mathcal{H}^{0,-\beta}(\mathbb{B})\to\mathcal{H}^{0,\beta}(\mathbb{B})^{*}\to H^{-1}(\mathbb{B})
$$
is continuous and compact, as $i_{\mathcal{H}^{0,\beta}(\mathbb{B})^{*},H^{-1}(\mathbb{B})}$
is a compact operator. The result follows now by the equality $T_{u}=i_{\mathcal{H}^{0,\beta}(\mathbb{B})^{*},H^{-1}(\mathbb{B})}\circ I\circ\widetilde{T}_{u}$.
\end{proof}
As $H_{0}^{1}(\mathbb{B})\hookrightarrow L^{4}(\mathbb{B})$, we can
define the Lyapunov function $\mathcal{L}:H_{0}^{1}(\mathbb{B})\to\mathbb{R}$
by \eqref{eq:Lyapunov}. It is a real analytic
function, as it is the composition of linear and multilinear functions. The following theorem is our main result of this section. Given $u_0\in X^0_{1,0}$, we denote by $\omega(u_0)$ the $\omega$-limit set $\omega_{X^0_{1,0}}(\{u_0\})$. 
\begin{thm}
\label{thm:LSpractical}Let $u_{0}\in X_{1,0}^{0}$.
If $\varphi\in\omega(u_{0})$, then there exist constants $c, \sigma>0$ and $\theta\in(0,1/2]$ such that the following inequality holds: 
$$
\left|\mathcal{L}(v)-\mathcal{L}(\varphi)\right|^{1-\theta}\le c\left\Vert D\mathcal{L}(v)\right\Vert _{H_{0}^{-1}(\mathbb{B})},
$$
whenever $\left\Vert v-\varphi\right\Vert _{H_{0}^{1}(\mathbb{B})}<\sigma$.
\end{thm}

\begin{proof}
The argument is standard and can be found e.g. in \cite{Chill} and \cite{Piotr}. We only stress here the necessary changes for the conical
singularities situation. 

We check the assumptions of Theorem \ref{thm:HarauxJendoubi}. For this, we choose $V=H_{0}^{1}(\mathbb{B})$,
$V^{*}=H_{0}^{-1}(\mathbb{B})$ and $H=\left\{ u\in\mathcal{H}^{0,0}(\mathbb{B});\int_{\mathbb{B}}ud\mu_{g}=0\right\} $, and define the function $E:H_{0}^{1}(\mathbb{B})\to\mathbb{R}$ by $E(v)=\mathcal{L}(v+\varphi)-\mathcal{L}(\varphi)$. It is clear
that $E(0)=0$. For the derivatives, for $v,h\in H_{0}^{1}(\mathbb{B})$ we have that
$$
\begin{aligned}D\mathcal{L}(v) & =-\Delta v+v^{3}-v-\fint_{\mathbb{B}}v^{3}d\mu_{g}\in H_{0}^{-1}(\mathbb{B}),\\
D^{2}\mathcal{L}(v)h & =-\Delta h+(3v^{2}-1)h-3\fint_{\mathbb{B}}v^{2}hd\mu_{g}\in H_{0}^{-1}(\mathbb{B}).
\end{aligned}
$$
The proof of the above expressions uses Theorem \ref{thm:Gauss-Theorem},
the Mellin-Sobolev embeddings from Corollary \ref{cor:Mellinsobolevembedding}
and the identification of Proposition \ref{prop:HzerobetadualH1}.

In order to prove that $DE(0)=0\in V^{*}$, we note that $DE(v)=D\mathcal{L}(v+\varphi)$.
Therefore
$$
DE(0)=-\Delta\varphi+\varphi^{3}-\varphi-(\varphi^{3})_{\mathbb{B}}.
$$
Since $\varphi\in\omega(u)$, we know that $\varphi\in H_{0}^{1}(\mathbb{B})$ is an
equilibrium point \cite[Theorem 8.4.6]{HJ}. Hence
$-\Delta\varphi+\varphi^{3}-\varphi$ is constant. In fact, as $\varphi\in\mathcal{D}(\Delta_{0}^{2})$ and $\frac{\partial\varphi}{\partial t}=0$, Theorem \ref{thm:Gauss-Theorem} with $u=v=\Delta\varphi-\varphi^{3}+\varphi$ and \eqref{eq:mainequation} shows that $\nabla(\Delta\varphi-\varphi^{3}+\varphi)=0$. This constant must
be equal to $(\varphi^{3})_{\mathbb{B}}$ by Theorem \ref{thm:Gauss-Theorem}
and $\int_{\mathbb{B}}\varphi d\mu_{g}=0$, which implies that $DE(0)=0$.

For showing Fredholm property of $D^{2}E(0)$, we first note that
$$
D^{2}E(0)h=-\Delta h+(3\varphi^{2}-1)h-3\fint_{\mathbb{B}}\varphi^{2}hd\mu_{g}\in H_{0}^{-1}(\mathbb{B}).
$$
The inclusion $H^{1}(\mathbb{B})\hookrightarrow H^{-1}(\mathbb{B})$
is compact and by Lemma \ref{prop:uemu2v} the map $h\in H^{1}(\mathbb{B})\mapsto v^{2}h\in H^{-1}(\mathbb{B})$
is also compact. In addition, the map $h\in H_{0}^{1}(\mathbb{B})\mapsto-3\fint_{\mathbb{B}}v^{2}hd\mu_{g}$
has finite rank and, therefore, it is also a compact operator from
$H_{0}^{1}(\mathbb{B})$ to $H_{0}^{-1}(\mathbb{B})$. We conclude
that $D^{2}E(0):H_{0}^{1}(\mathbb{B})\to H_{0}^{-1}(\mathbb{B})$
is a compact perturbation of the isomorphism $\Delta:H_{0}^{1}(\mathbb{B})\to H_{0}^{-1}(\mathbb{B})$.
\end{proof}
As $\mathcal{L}$ is a Lyapunov function bounded from below and as
$\omega(u_{0})$ is compact in $\mathcal{H}^{2,\gamma+2}(\mathbb{B})\oplus\mathbb{C}_{\omega}\hookrightarrow H^{1}(\mathbb{B})$,
we conclude the following:
\begin{cor}
\label{cor:LSappl}Let $u_{0}\in X_{1,0}^{0}$.
Then\\
\emph{(i)} There is a constant $\mathcal{L}_{\infty}\in\mathbb{R}$ such that
$\mathcal{L}(\varphi)=\mathcal{L}_{\infty}$, for all $\varphi\in\omega(u_{0})$.\\
\emph{(ii)} There is a neighborhood $U\subset H_{0}^{1}(\mathbb{B})$ of $\omega(u_{0})$
and constants $C>0$, $\theta\in(0,1/2]$ such that
$$
\left|\mathcal{L}(v)-\mathcal{L}_{\infty}\right|^{1-\theta}\le C\left\Vert D\mathcal{L}(v)\right\Vert _{H_{0}^{-1}(\mathbb{B})},
$$
for all $v\in U$.
\end{cor}

Having the Lojasiewicz-Simon inequality, we can prove the convergence theorem below.
\begin{prop}
Let $u_{0}\in X_{1,0}^{0}$
and $u$ be the solution of \eqref{eq:mainequation}.
Then there exists a $u_{\infty}\in\omega(u_{0})$ such that $\lim_{t\to\infty}\left\Vert u(t)-u_{\infty}\right\Vert _{H^{-1}_0(\mathbb{B})}=0$.
\end{prop}

\begin{proof}
The proof follows from the arguments in \cite[Section 3]{Chill}. We sketch here the steps for the
convenience of the reader.

Let $\mathcal{L}_{\infty}=\lim_{t\to\infty}\mathcal{L}(u(t))$. Then
$\mathcal{L}_{\infty}\le\mathcal{L}(u(t))$ for all $t\in[0,\infty)$.
We define the function $H:[0,\infty)\to\mathbb{R}$ by
$$
H(t)=(\mathcal{L}(u(t))-\mathcal{L}_{\infty})^{\theta},
$$
where $\theta\in(0,1/2]$ as in Corollary \ref{cor:LSappl}. The function
$H$ is non-negative and non-increasing, as
$$
\frac{d}{dt}H(t)=\theta(\mathcal{L}(u(t))-\mathcal{L}_{\infty})^{\theta-1}\frac{d}{dt}\mathcal{L}(u(t))\le0.
$$
Moreover $\lim_{t\to\infty}H(t)=0$. Let $U\subset H_{0}^{1}(\mathbb{B})$
be the open set of Corollary \ref{cor:LSappl}. Due to \cite[Theorem 5.1.8]{HJ}, we have
\begin{equation}\label{eq:convomega}
\lim_{t\to\infty}(\inf_{v\in\omega(u_{0})}\left\Vert T(t)u-v\right\Vert _{\mathcal{H}^{2,\gamma+2}(\mathbb{B})\oplus\mathbb{R}_{\omega}})=0.
\end{equation}
Thus, there exists $t_{0}>0$ such that, for
$t\ge t_{0}$, we have $u(t)\in U$. Hence, for $t\ge t_{0}$, we
estimate
\begin{eqnarray*}-\frac{d}{dt}H(t) &=&\theta(\mathcal{L}(u(t))-\mathcal{L}_{\infty})^{\theta-1}\left(-\frac{d}{dt}\mathcal{L}(u(t))\right)\\
 &\overset{(1)}{\ge}& C\frac{\int_{\mathbb{B}}\left\langle \nabla(-\Delta u+u^{3}-u),\nabla(-\Delta u+u^{3}-u)\right\rangle _{g}d\mu_{g}}{\left\Vert -\Delta u+u^{3}-u-(u^{3})_{\mathbb{B}}\right\Vert _{H_{0}^{-1}(\mathbb{B})}}\\
 & \overset{(2)}{\ge} & C\frac{\int_{\mathbb{B}}\left\langle \nabla(-\Delta u+u^{3}-u),\nabla(-\Delta u+u^{3}-u)\right\rangle d\mu_{g}}{\left\Vert \Delta(\Delta u-u^{3}+u)\right\Vert _{H_{0}^{-1}(\mathbb{B})}}\\
 & \overset{(3)}{\ge} & C\frac{\left\Vert \Delta(-\Delta u+u^{3}-u)\right\Vert _{H_{0}^{-1}(\mathbb{B})}^{2}}{\left\Vert \Delta(\Delta u-u^{3}+u)\right\Vert _{H_{0}^{-1}(\mathbb{B})}}=C\left\Vert \Delta(-\Delta u+u^{3}-u)\right\Vert _{H_{0}^{-1}(\mathbb{B})}.
\end{eqnarray*}
In $(1)$ we have used Theorem \ref{thm:Gauss-Theorem} and Corollary \ref{cor:LSappl}, in $(2)$ Corollary \ref{lem:EquivalenceofnormDeltauandu-1} and in
$(3)$ the definition of $\|\cdot\|_{H_{0}^{-1}(\mathbb{B})}$
and the isomorphism $\Delta:H_{0}^{1}(\mathbb{B})\to H_{0}^{-1}(\mathbb{B})$.
By \eqref{eq:mainequation}, we infer that
$$
\left\Vert \frac{\partial u}{\partial t}\right\Vert _{H_{0}^{-1}(\mathbb{B})}=\left\Vert \Delta(-\Delta u+u^{3}-u)\right\Vert _{H_{0}^{-1}(\mathbb{B})}\le-C\frac{d}{dt}H(t).
$$
Hence $\partial_{t}u\in L^{1}(0,\infty;H_{0}^{-1}(\mathbb{B}))$
and 
$$
u_{\infty}:=\lim_{t\to\infty}u(t)=u(0)+\int_{0}^{\infty}\frac{\partial u}{\partial t}(s)ds,
$$
where the limit is taken in $H_{0}^{-1}(\mathbb{B})$.

It remains to prove that $u_{\infty}\in\omega(u_{0})$. We know that
$\omega(u_{0})\subset\mathcal{H}^{2,\gamma+2}(\mathbb{B})\oplus\mathbb{R}_\omega$
is compact. As $\mathcal{H}^{2,\gamma+2}(\mathbb{B})\oplus\mathbb{R}_\omega\hookrightarrow H^{-1}(\mathbb{B})$,
we conclude that $\omega(u_{0})$ is also compact in $H_{0}^{-1}(\mathbb{B})$. Thus it is also closed. Equation \eqref{eq:convomega} implies that 
$$
\lim_{t\to\infty}(\inf_{v\in\omega(u_{0})}\left\Vert T(t)u-v\right\Vert _{H^{-1}(\mathbb{B})})=0
$$
 and, therefore, that $\inf_{v\in\omega(u_{0})}\left\Vert u_{\infty}-v\right\Vert _{H^{-1}(\mathbb{B})}=0$.
As $\omega(u_{0})$ is closed in $H_{0}^{-1}(\mathbb{B})$, we conclude
that $u_{\infty}\in\omega(u_{0})$.
\end{proof}

Finally, we prove part (ii) of Theorem \ref{thm:MainTheorem}.

\begin{proof} (of part (ii) of Theorem \ref{thm:MainTheorem})
Recall first that $T(t)u_{0}\in\cap_{r\ge0}\mathcal{D}(\Delta_{r}^{2})$ due to \eqref{extrareg}.
Therefore the limit in $\mathcal{D}(\Delta_{r}^{2})$
makes sense.
We know by Theorem \ref{thm:Attractorestimate} that $\left\{T(t)u_{0}\right\} _{t\ge T}$ is precompact
in $\mathcal{D}(\Delta_{r}^{2})$ for some $T>0$. Let $u_{\infty}\in\omega(u_{0})$
be such that $\lim_{t\to\infty}T(t)u_{0}=u_{\infty}$
in $H_{0}^{-1}(\mathbb{B})$. Note that since $u_{\infty}$ is an equilibrium point \cite[Theorem 8.4.6]{HJ}, $T(t)u_{\infty}=u_{\infty}$ for all $t>0$ and, hence, $u_{\infty}\in\cap_{r\ge0}\mathcal{D}(\Delta_{r}^{2})$.
Suppose that this limit does not hold
in $\mathcal{D}(\Delta_{r}^{2})$.
Then, there exist an $\varepsilon_{0}>0$ and a sequence $t_{k}\to\infty$ such that $\left\Vert T(t_{k})u_{0}-u_{\infty}\right\Vert _{\mathcal{D}(\Delta_{r}^{2})}\ge\varepsilon_{0}.$
By compacteness, there exists a subsequence $t_{k_{j}}\to \infty$
such that $T(t_{k_{j}})u_{0}$ converges to a function $w$ in $\mathcal{D}(\Delta_{r}^{2})$.
This implies that $T(t_{k_{j}})u_{0}$ also converges to $w$ in $H^{-1}(\mathbb{B})$.
Therefore $w=u_{\infty}$ by uniqueness of the limit and we obtain a contradiction.
\end{proof}

\section{Acknowledgements}

We would like to express our sincere gratitude to the anonymous referee for the careful reading of the manuscript and thoughtful comments.

Pedro T. P. Lopes was partially supported by FAPESP 2019/15200-1.


\begin{thebibliography}{99}

\bibitem{Am} H. Amann. {\em Linear and quasilinear parabolic problems}. Monographs in Mathematics {\bf 89}, Birkh\"auser Verlag (1995).

\bibitem{Arendtetal} W. Arendt, C. J. Batty, M. Hieber, F. Neubrander {\em Vector-valued Laplace transforms and Cauchy problems}. Birkh\"auser, Basel, (2001).

\bibitem{Chill} R. Chill, E. Fasangova, J. Pruess. {\em Convergence to steady states of solutions of the Cahn--Hilliard and Caginalp equations with dynamic boundary conditions}. Mathematische Nachrichten {\bf 279}, no. 13-14, 1448--1462 (2006).

\bibitem{CD} J. W. Cholewa, T. Dlotko, N. Chafee. {\em Global attractors in abstract parabolic problems}. Cambridge University Press {\bf 278} (2000).

\bibitem{CSS1} S. Coriasco, E. Schrohe, J. Seiler. {\em Differential operators on conic manifolds: Maximal regularity and parabolic equations.} Bull. Soc. Roy. Sci. Li\`ege {\bf 70}, no. 4-6, 207--229 (2001). 

\bibitem{Evans} L. Evans. {\em Partial differential equations, second edition}. Graduate Studies in Mathematics {\bf 19}, AMS (2010).

\bibitem{GKM} J. Gil, T. Krainer, G. Mendoza. {\em Resolvents of elliptic cone operators}. J. Funct. Anal. {\bf 241}, no. 1, 1--55 (2006).

\bibitem{GilMendoza} J. Gil, G. Mendoza. {\em Adjoints of elliptic cone operators.} American Journal of Mathematics {\bf 125}, no. 2, 357--408 (2003).

\bibitem{HJ} A. Haraux, M. A. Jendoubi. {\em The convergence problem for dissipative autonomous systems: classical methods and recent advances}. Springer Verlag (2015).

\bibitem{Le} M. Lesch. {\em Operators of Fuchs type, conical singularities, and asymptotic methods}. Teubner-Texte zur Mathematik {\bf136}, Teubner Verlag (1997).

\bibitem{LopesRoidos} P. T. P. Lopes, N. Roidos. {\em Smoothness and long time existence for solutions of the Cahn-Hilliard equation on manifolds with conical singularities}. Monatshefte f\"ur Mathematik {\bf 197}, no. 4, 677--716 (2022).

\bibitem{Lunardi} A. Lunardi. {\em Interpolation theory}. Lecture Notes Scuola Normale Superiore {\bf 16}, Pisa (2018).

\bibitem{Miranville} A. Miranville. {\em The Cahn--Hilliard equation: Recent advances and applications}. SIAM (2019).

\bibitem{Pazy} A. Pazy. {\em Semigroups of linear operators and applications to partial differential equations}. Applied Mathematical Sciences {\bf 44}. Springer Science \& Business Media (2012).

\bibitem{Robinson} J. Robinson. {\em Infinite-dimensional dynamical systems. An introduction to dissipative parabolic PDEs and the theory of global attractors}. Cambridge Texts in Applied Mathematics, Cambridge University Press (2001).

\bibitem{RS2} N. Roidos, E. Schrohe. {\em Bounded imaginary powers of cone differential operators on higher order Mellin-Sobolev spaces and applications to the Cahn-Hilliard equation}. J. Differential Equations {\bf 257}, no. 3, 611--637 (2014).

\bibitem{RS3} N. Roidos, E. Schrohe. {\em Existence and maximal L p-regularity of solutions for the porous medium equation on manifolds with conical singularities}. Comm. Partial Differential Equations {\bf 41.9}, 1441--1471. (2016).

\bibitem{RS1} N. Roidos, E. Schrohe. {\em The Cahn-Hilliard equation and the Allen-Cahn equation on manifolds with conical singularities}. Comm. Partial Differential Equations {\bf 38}, no. 5, 925--943 (2013).

\bibitem{Piotr} P. Rybka, K.-H. Hoffnlann. {\em Convergence of solutions to Cahn-Hilliard equation}. Comm. Partial Differential Equations {\bf 24}, no. 5-6, 1055--1077 (1999).

\bibitem{SS1} E. Schrohe, J. Seiler. {\em Bounded $H_{\infty}$-calculus for cone differential operators}. J. Evol. Equ. {\bf 18}, no 3, 1395--1425 (2018).

\bibitem{SS} E. Schrohe, J. Seiler. {\em The resolvent of closed extensions of cone differential operators}. Can. J. Math. {\bf 57}, no. 4, 771--811 (2005).

\bibitem{Schu2} B.-W. Schulze. {\em Boundary value problems and singular pseudo-differential operators}. Pure and Applied Mathematics: A Wiley Series of Texts, Monographs, and Tracks, Wiley (1998). 

\bibitem{Schu} B.-W. Schulze. {\em Pseudo-differential operators on manifolds with singularities}. Studies in Mathematics and Its Applications {\bf 24}, North Holland (1991).

\bibitem{Simon} L. Simon. {\em Asymptotics for a class of non-linear evolution equations, with applications to geometric problems}. Annals of Mathematics, 525--571 (1983).

\bibitem{Song} L. Song, Y. Zhang, T. Ma. {\em Global attractor of the Cahn--Hilliard equation in H-k spaces}. Journal of mathematical analysis and applications {\bf 355}, no. 1, 53--62 (2009).

\bibitem{Temam} R. Temam. {\em Infinite-dimensional dynamical systems in mechanics and physics}. Applied Mathematical Science {\bf 68}, Springer Science and Business Media, Springer (2012).

\end{thebibliography}
\end{document}